\documentclass[12pt]{article}
\usepackage{amsmath,amssymb,amsthm}
\usepackage{color,url,array,multirow,graphicx}
\usepackage{tikz}
\usetikzlibrary{arrows}

\newtheorem{theorem}{Theorem}
\newtheorem{corollary}[theorem]{Corollary}
\newtheorem{construction}[theorem]{Construction}

\theoremstyle{definition}

\newtheorem{question}{Question}

\newcommand{\oh}{\mathcal{O}}

\newcommand{\cF}{\mathcal{F}}

\newcommand{\bs}{\mathbf{s}}
\newcommand{\bt}{\mathbf{t}}
\newcommand{\bu}{\mathbf{u}}
\newcommand{\bv}{\mathbf{v}}
\newcommand{\bx}{\mathbf{x}}
\newcommand{\by}{\mathbf{y}}
\newcommand{\bz}{\mathbf{z}}

\newcommand{\bZ}{\mathbb{Z}}

\begin{document}
\title{Block-avoiding point sequencings}
\author{Simon R. Blackburn\\
Department of Mathematics\\
Royal Holloway University of London\\
Egham, Surrey TW20 0EX, United Kingdom\\
{\tt s.blackburn@rhul.ac.uk}
\and
Tuvi Etzion\thanks{Supported in part by the Israeli Science Foundation
under grant no. 222/19}\\
Department of Computer Science\\
Technion\\
Haifa 3200003, Israel\\
{\tt etzion@cs.technion.ac.il}
}
\maketitle\newpage

\begin{abstract}
Let $n$ and $\ell$ be positive integers. Recent papers by Kreher, Stinson and Veitch have explored variants of the problem of ordering the points in a triple system (such as a Steiner triple system, directed triple system or Mendelsohn triple system) on $n$ points so that no block occurs in a segment of $\ell$ consecutive entries (thus the ordering is locally block-avoiding). We describe a greedy algorithm which shows that such an ordering exists, provided that $n$ is sufficiently large when compared to $\ell$. This algorithm leads to improved bounds on the number of points in cases where this was known, but also extends the results to a significantly more general setting (which includes, for example, orderings that avoid the blocks of a design). Similar results for a cyclic variant of this situation are also established.

We construct Steiner triple systems and quadruple systems where $\ell$ can be large, showing that a bound of Stinson and Veitch is reasonable. Moreover, we generalise the Stinson--Veitch bound to a wider class of block designs and to the cyclic case.

The results of Kreher, Stinson and Veitch were originally inspired by results of Alspach, Kreher and Pastine, who (motivated by zero-sum avoiding sequences in abelian groups) were interested in orderings of points in a partial Steiner triple system where no segment is a union of disjoint blocks. Alspach~\emph{et al.}\ show that, when the system contains at most $k$ pairwise disjoint blocks, an ordering exists when the number of points is more than $15k-5$. By making use of a greedy approach, the paper improves this bound to $9k+\oh(k^{2/3})$.
\end{abstract}

\section{Introduction}

Let $V$ be a finite set of cardinality $n$. A sequence $x_1,x_2,\ldots,x_n$ over~$V$ is a \emph{sequencing} if the elements $x_i$ are a permutation of the elements of~$V$. A \emph{segment} of a sequence is a subsequence of consecutive entries.

Suppose that $V$ is the set of points of a Steiner triple system (STS). Recall that a STS is a set of $3$-subsets, \emph{blocks}, of~$V$ such that every pair of points is contained in a unique block. Several recent papers have explored variations on the idea of sequencings that are block-avoiding. For example, Stinson and Kreher~\cite{KreherStinson_AJC_STS} define a sequencing to be \emph{$\ell$-good} (for some integer~$\ell$) if no segment of length $\ell$ contains a block. An STS is $\ell$-good if it possesses an $\ell$-good sequencing. Stinson and Kreher show that any STS with $n>3$ is $3$-good, and any STS with $n>71$ is $4$-good. More generally, Stinson and Veitch~\cite{StinsonVeitch_STS_sequencing} show that an STS with $n>\ell^6/16+\oh(\ell^5)$ is $\ell$-good. They use a greedy algorithm to establish this result. The naive greedy algorithm chooses the elements $x_i$ of the sequencing in order, making sure the elements are distinct and avoiding the points that would create a block when combined with some of the $\ell-1$ most recently chosen elements of the sequence. But this naive algorithm might fail towards the end of the sequencing, because the elements that have not yet been used all form a block with recent elements. So Stinson and Veitch carefully design the beginning of the sequence so that any problematic elements can be swapped with elements there, allowing the algorithm to complete. 

Kreher, Stinson and Veitch~\cite{KreherStinsonVeitch_DTS,KreherStinsonVeitch_smallDTS,KreherStinsonVeitch_Mendelsohn} have explored other variants of this problem, including the cyclic $\ell$-good property (where segments can `wrap around' from the end to the start of the sequence) and variants where the triples are ordered in some way (Mendelsohn triple systems, and directed triple systems). Using greedy algorithms, all these variants of the $\ell$-good property can be shown to be feasible when $n$ is sufficiently large compared to $\ell$. 

All of the results above also hold in the more general situation of \emph{partial} Steiner triple systems: a PSTS is a set of $3$-subsets, \emph{blocks}, of $V$ that pairwise intersect in at most one element.

A major aim of this paper is to describe a greedy algorithm which provides sequencings for all these variants of the $\ell$-good property and more, and which improves on the greedy algorithms in the papers above (in the sense that it succeeds for smaller values of $n$). To aid understanding, in Section~\ref{sec:PSTS} we first describe the algorithm in the special case of partial Steiner triple systems. We show that any PSTS  with $n>3\ell^4/4+\oh(\ell^3)$ is $\ell$-good; a slightly more careful argument given in the section which follows shows that the bound may be improved to $n>\ell^4/2+\oh(\ell^3)$. In Section~\ref{sec:general} we set out the general problem and describe an algorithm to solve it. Some consequences of this general algorithm, including results on $\ell$-good sequencings of Mendelsohn triple systems and directed triple systems, are also provided there. We consider the cyclic case in Section~\ref{sec:cyclic}. The key new ingredient in all these algorithms is to design the first part of the sequencing with some elements unspecified, that are later filled by the elements that remain after the naive greedy algorithm becomes stuck. We make sure that all elements that cause problems are used in the early part of the sequencing. We then use the naive greedy algorithm to construct the end of the sequencing, before finalising the start of the sequencing using the elements that remain. 

Rather than asking for values of $n$ and $\ell$ such that \emph{all} Steiner triple systems on $n$ points have an $\ell$-good sequencing, it is natural to ask for \emph{specific examples} of Steiner triple systems with an $\ell$-good sequencing, where $\ell$ is large. In this vein, Stinson and Veitch~\cite[Theorem~2.1]{StinsonVeitch_STS_sequencing} have shown that when an STS on $n$ points has an $\ell$-good sequencing, then $\ell\leq (n+2)/3$. In Section~\ref{sec:STS}, we construct an infinite family of $n$-point Steiner triple systems with $\ell$-good sequencings, where $\ell=(n+3)/4$. These examples show that the Stinson--Veitch bound is reasonable: they are the first examples of $\ell$-good sequencings with $\ell$ linear in $n$. Section~\ref{sec:STS} also contains results in the cyclic case. Indeed, the examples earlier in the section have the cyclic $(n+3)/4$-good property; we show that the Stinson--Veitch bound can be improved in the cyclic case to $\ell\leq 0.329n+\oh(1)$.

Recall that a \emph{Steiner quadruple system} is a set of $4$-subsets (\emph{blocks}) of a set~$V$ with the property that every $3$-subset of $V$ is contained in a unique block. In Section~\ref{sec:SQS}, we consider $\ell$-good sequencings of Steiner quadruple systems.  We construct an infinite family of $n$-point Steiner quadruple systems with (cyclic) $\ell$-good sequencings, where $\ell=n/4+1$, and we extend the Stinson--Veitch bound to show that $\ell\leq \left(1/\sqrt{6}+o(1)\right)n$. We also provide an upper bound on $\ell$ that holds for all block designs. 

In Section~\ref{sec:alspach} we turn to a closely related topic, which inspired the results of Kreher, Stinson and Veitch above. Alspach, Kreher and Pastine~\cite{AlspachKreherPastine} considered the following problem, motivated by the properties of zero-sum free sequences in abelian groups~\cite{Alspach,AlspachLiversidge}.  A triple system is \emph{sequenceable} if there exists a sequencing with no non-empty segment having elements equal to a union of disjoint blocks. Kreher and Stinson~\cite{KreherStinson_ICA_STS} provide infinitely many examples of PSTS that are not sequenceable. However, Aspach \emph{et al.} show that a PSTS with at most $3$ pairwise disjoint blocks is sequenceable, and that if a PSTS has at most $k$ disjoint blocks then it is sequenceable whenever $n \geq 15k-5$. We modify their method, treating large and small segments differently, and show that this bound may be improved to $n\geq 9k-\oh(k^{2/3})$.

Finally, in Section~\ref{sec:comments}, we provide some open problems and topics for further research.

\section{Sequenceability of partial Steiner triple systems}
\label{sec:PSTS}

We begin by defining and recapping notation. Let $V$ be a finite set of cardinality $n$. We say that a sequence $x_1,x_2,\ldots ,x_r$ is a \emph{partial sequencing} (of~$V$) if the elements $x_i\in V$ are distinct. If $r=n$, we say that  $x_1,x_2,\ldots ,x_n$ is a \emph{sequencing}. For a partial Steiner triple system (PSTS) with point set $V$, we say that a sequence $x_1,x_2,\ldots ,x_r$ of points is \emph{good} if $\{x_1,x_2,\ldots ,x_r\}$ does not contain a block. We say that a (partial) sequencing is \emph{$\ell$-good} if all segments of length $\ell$ or less are good.

\begin{theorem}
\label{thm:main_PSTS}
Let $\ell$ be a positive integer. Whenever
\begin{equation}
\label{eqn:STS}
n\geq \left(2\ell+3\binom{\ell-1}{2}\right)\binom{\ell-1}{2}+\ell=3\ell^4/4+\oh(\ell^3),
\end{equation}
all partial Steiner triple systems of order $n$ have $\ell$-good sequencings.
\end{theorem}
\begin{proof}
Since every sequencing is $1$-good and $2$-good, we may assume that $\ell\geq 3$. 
Let $n$ be an integer such that~\eqref{eqn:STS} holds. Fix a PSTS with point set~$V$, where $|V|=n$. Set $L=\binom{\ell-1}{2}$, so $n\geq \left(
2\ell+3L\right)L+\ell$. To prove the theorem, it suffices to construct a sequencing for our PSTS. The following algorithm constructs such a sequencing.

\paragraph{Stage~1:} Greedily construct an $\ell$-good partial sequencing $x_1,x_2,\ldots,x_{(\ell-1)L}$. So at each stage we choose the point $x_i\in V$ so that the following two conditions are met:
\begin{itemize}
\item[(i)] $x_i$ is distinct from $x_1,x_2,\ldots,x_{i-1}$, and
\item[(ii)] there is no block in the PSTS contained in the set $\{x_r,x_{r+1},\ldots ,x_i\}$, where $r=\max\{1,i-\ell+1\}$.
\end{itemize}
Note that $i-1$ points $x_1,x_2,\ldots ,x_{i-1}$ are ruled out as choices for $x_i$ by condition~(i). We claim that at most $L$ points ruled out as choices for $x_i$ by condition~(ii). To see this, first note that any block contained in $\{x_r,x_{r+1},\ldots ,x_i\}$ must involve $x_i$, since we may assume by induction that $x_1,x_2,\ldots,x_{i-1}$ is $\ell$-good. So the block must intersect $\{x_r,x_{r+1},\ldots ,x_{i-1}\}$ in a subset of size $2$. Once the two points in this subset are chosen, the third point in the block is determined. So there are at most $L$ such blocks, as there are $L$ choices for the two points that lie in $\{x_r,x_{r+1},\ldots ,x_{i-1}\}$. Each block rules out at most one choice for $x_i$, and so our claim follows. Since $n\geq(\ell-1)L+L> i-1+L$, the greedy algorithm will always succeed.

We divide the resulting sequence into $L$ segments $\bx^1,\bx^2,\ldots ,\bx^L$, each of length $\ell-1$. So $\bx^j=x_{(j-1)(\ell-1)+1},x_{(j-1)(\ell-1)+2},\ldots ,x_{(j-1)(\ell-1)+\ell-1}$.

Let $V'=V\setminus\{x_1,x_2,\ldots,x_{(\ell-1)L}\}$, so
\[
|V'|=n-(\ell-1)L.
\]
The initial segment of the sequencing we will construct will have the form $y_1\bx^1y_2\bx^2\cdots y_L\bx^L$, for distinct points $y_1,y_2,\ldots,y_L\in V'$ that we have not yet specified. Of course, not all choices of points $y_j$ will preserve the $\ell$-good property of the sequence. We say that a point $y\in V'$ is \emph{unfortunate} if the $\ell$-good property of the sequencing fails in this way, and define $U$ to be the set of unfortunate points.  Since being $\ell$-good is a property of segments of length at most $\ell$, and the segments $\bx^i$ are all of length $\ell-1$, we see that $U$ is the set of points $y\in V'$ such that for some $i\in\{1,2,\ldots,L-1\}$ the sequence $\bx^iy\bx^{i+1}$ is not $\ell$-good. (Note that we have covered the case that $y\bx^1$ fails to be $\ell$-good, since this case implies that $\bx^1y$, and so $\bx^1y\bx^{2}$, fails to be $\ell$-good.) There are $L-1$ choices for $i$. There are at most $(\ell-1)(\ell-2)+\binom{\ell-1}{2}$ pairs of positions in $\bx^i$ or $\bx^{i+1}$ that are contained in a segment of $\bx^iy\bx^{i+1}$ of length at most $\ell$. Each pair of positions gives rise to at most one unfortunate element (the third point in the block containing the points in these positions, if such a block exists). So 
\begin{equation}
\label{eqn:USTS}
|U|\leq \tfrac{3}{2}(\ell-1)(\ell-2)(L-1)\leq 3L^2.
\end{equation}

Note that the partial sequencing $y_1\bx^1y_2\bx^2\cdots y_L\bx^L$ is $\ell$-good whenever the points $y_i$ are distinct, and none of the points $y_i$ are unfortunate.

Our strategy in stages~2 to~4 will be to construct a partial sequencing $\bz$ over $V'$ such that: all but $L$ points of $V'$ are used; all unfortunate points are used; and $\bx^L\bz$ is $\ell$-good. Once we have accomplished this, we set $y_1,y_2,\ldots,y_L$ to be the $L$ points that do not appear in $\bz$. The sequence $y_1\bx^1y_2\bx^2\cdots y_L\bx^L\bz$ is an $\ell$-good sequencing, as required.

\paragraph{Stage~2:} If $|U|>L$, greedily find a partial sequencing $\bu$ of $|U|-L$ unfortunate points, so that the partial sequencing $\bx^L\bu$ is $\ell$-good. Note that this is possible, since at each stage maintaining the $\ell$-good condition rules out at most $L$ points, and there are always more than $L$ unfortunate points available.

\paragraph{Stage~3:} Add each of the remaining unfortunate points to our partial sequencing $\bu$ (maintaining the $\ell$-good property) as follows. For each remaining $u\in U$, extend the partial sequencing by $z_1z_2\cdots z_{\ell-1}u$, where $z_1,z_2,\ldots,z_{\ell-1}\in V'$ are chosen greedily so that the $\ell$-good property is maintained and also so that $\{z_1,z_2,\ldots,z_{\ell-1},u\}$ does not contain a block. Note that when we choose an element $z_j$, we have previously used fewer than $|U|+(\ell-1)L$ elements of $V'$. Moreover, at most $L$ elements are ruled out to maintain the $\ell$-good property, and at most $\ell-2$ elements by the fact that $z_j$ cannot be in a block with $u$ and one of $z_1,z_2,\ldots,z_{j-1}$. Now~\eqref{eqn:STS} and~\eqref{eqn:USTS} imply that
\[
|V'|-(|U|+L+\ell-1)\geq n-((\ell-1)L+3L^2+L+\ell-1)>0,
\]
and so the greedy algorithm always succeeds.

After Stage~3, we have produced a partial sequencing $\bv$ which extends $\bu$. All unfortunate elements appear in $\bv$, and $\bx^L\bv$ is $\ell$-good. Moreover, since we have used at most $|U|+L\ell$ points from $V'$, we see that~\eqref{eqn:USTS} implies at least $L$ points in $V'$ have not yet been used:
\[
|V'|-(|U|+L\ell)=n-((\ell-1)L+3L^2+L\ell)>L.
\]

\paragraph{Stage~4:} Use the greedy algorithm to extend our partial $\ell$-good sequencing $\bv$ until just $L$ unused elements of $V'$ remain. Again, we require that the resulting sequencing $\bz$ has the property that $\bx^L\bz$ is $\ell$-good. Since there are always more than $L$ points from $V'$ that are have not yet been used, the greedy algorithm always succeeds in constructing $\bz$.
 
At the end of Stage~4, the partial sequencing $\bz$ contains $n-L$ points of~$V'$. Let $y_1,y_2,\ldots ,y_L$ be the $L$ remaining points in $V'$. Then the sequence $y_1\bx^1y_2\bx^2\cdots y_L\bx^L\by$ is an $\ell$-good sequencing, as desired.
\end{proof}

\mbox{}

We comment that the lower bound on $n$ can be improved a little. First, slightly better estimates could be used. For example, the bound on the number of unfortunate points can be improved. Indeed, Theorem~\ref{thm:general} below does provide a slightly improved bound when specialised to the PSTS case. Secondly, the algorithm itself could be improved, by taking unfortunate elements arising from the start of the sequence $x_1,x_2,\ldots,x_{(\ell-1)L}$ and including them towards the end of this sequence. But even with these changes the lower bound on $n$ is still of the order of~$\ell^4$: new methods are needed to significantly improve the bound (which is dominated by the size of the set of unfortunate elements).

\section{The general (non-cyclic) case}
\label{sec:general}

When we are interested in $\ell$-good sequencings of partial Steiner triple systems, we first choose a positive integer $n$ and a set $V$ of cardinality $n$. We then choose a PSTS over $V$, and declare a segment to be good if the set of its elements does not contain a block. To rephrase, we define a set $\cF$ of (`forbidden') sequences over $V$, namely the sequences of length~$3$ whose elements form a block, and we declare a sequence over $V$ to be good if none of its subsequences lie in $\cF$. We then define a sequence $x_1,x_2,\ldots,x_n$ of length $n$ to be an $\ell$-good sequencing if the elements $x_i$ of the sequence are distinct and if the segments $x_i,x_{i+1},\ldots,x_{i+\ell-1}$ for $1\leq i\leq n-\ell+1$ are all good. We first generalise this process, and extract the properties we need for our greedy algorithm to work.

Let $V$ be a finite set of cardinality $n$. As before, we say that a sequence $x_1,x_2,\ldots,x_r$ over $V$ is a \emph{partial sequencing} if the elements $x_i$ are distinct, and a \emph{sequencing} if in addition $r=n$.

Let $\cF$ be a set of finite sequences over $V$. We say that a sequence is 
\emph{$\cF$-good} if none of its subsequences lie in $\cF$. We say that a (partial) sequencing is an \emph{$(\ell,\cF)$-good (partial) sequencing} if the elements $x_i$ of the sequence are distinct and if all segments of the sequence of length $\ell$ or less are $\cF$-good.

We note that when $\cF$ is the set of sequences of length $3$ whose elements form blocks of a partial Steiner triple system the notions of $(\ell,\cF)$-good sequencing and $\ell$-good sequencing are identical.


Let $\ell$ be a positive integer. For an integer $L$ (possibly depending on $\ell$), we say that $\cF$ has the \emph{$L$-suffix property} if the following statement holds. For any non-negative integer $r$ with $r\leq \ell-1$ and any $\cF$-good sequence $x_1,x_2,\ldots,x_{r}$ over $V$, there are at most $L$ choices for $x\in V$ such that the sequence $x_1,x_2,\ldots,x_r,x$ fails to be $\cF$-good. Similarly, for an integer $L'$, we define $\cF$ to have the \emph{$L'$-prefix property} if there are at most $L'$ choices for $x\in V$ such that the sequence $x, x_1,x_2,\ldots,x_r$ fails to be $\cF$-good. We note that in the PSTS case, the choices for $x$ where $x_1,x_2,\ldots,x_r,x$ fails to be $\cF$-good are precisely the third points in the blocks intersecting $\{x_1,x_2,\ldots ,x_r\}$ in two distinct points; since there are at most $\binom{\ell-1}{2}$ such blocks we see that $\cF$ has the $\binom{\ell-1}{2}$-suffix property. Similarly, $\cF$ has the $\binom{\ell-1}{2}$-prefix property in this case.

For some sets of sequences $\cF$ (for example, those sets of sequences that are closed under permuting their elements) for any sequence $x_1,x_2,\ldots,x_{r}$ the sets of elements $x$ that are counted by the $L$-suffix and $L$-prefix property are equal. In this situation, we say that $\cF$ is \emph{symmetric}. So in the PSTS case, $\cF$ is symmetric.

For an integer $K$, possibly depending on $\ell$, we say that $\cF$ has the \emph{$K$-insertion property} if the following statement holds. For any $(\ell,\cF)$-good sequence $x_1,x_2,\ldots,x_{2\ell-2}$ over $V$, there are at most $K$ choices for $x\in V$ such that the sequences $x_1,x_2,\ldots,x_{\ell-1},x$ and $x,x_\ell,x_{\ell+1},\ldots,x_{2\ell-2}$ are $(\ell,\cF)$-good but the sequence $x_1,x_2,\ldots,x_{\ell-1},x,x_\ell,x_{\ell+1},\ldots,x_{2\ell-2}$ fails to be $(\ell,\cF)$-good. In the partial STS case, $\cF$ has the $\binom{\ell-1}{2}$-insertion property. To see this, note that there are $\binom{\ell-1}{2}$ ways of choosing positions $i\in\{1,2,\ldots,\ell-1\}$ and $j\in\{\ell,\ell+1,\ldots,2\ell-2\}$ with $j-i<\ell-1$. There is at most one element $x\in V$ such that $\{x_i,x_j,x\}$ is a block, and all the elements $x$ we are counting arise in this way. 

For integers $J$ and $s$, possibly depending on $\ell$, we say that $\cF$ has the \emph{$(J,s)$-reachability property} if the following statement holds. Let $W\subseteq V$ with $|W|\geq J$.  Let $x_1,x_2,\ldots,x_{\ell-1}$ be a partial $(\ell,\cF)$-good sequencing whose elements $x_i$ do not lie in $W$. Let $w\in W$. Then there exist elements $w_1,w_2,\ldots,w_s\in W$ such that $x_1,x_2,\ldots,x_{\ell-1},w_1,w_2,\ldots,w_s,w$ is a partial $(\ell,\cF)$-good sequencing. In the PSTS case, we claim that $\cF$ has the $(\binom{\ell-1}{2}+2\ell,\ell-1)$-reachability property. For we may greedily choose distinct elements $w_i\in W\setminus\{w\}$ such that
\begin{itemize}
\item The partial sequencing $x_1,x_2,\ldots,x_{\ell-1},w_1,w_2,\ldots,w_i$ is $\ell$-good, and
\item There is no block of the form $\{w_j,w_i,w\}$ for $1\leq j<i$.
\end{itemize}
At the point when the greedy algorithm chooses $w_i$, maintaining the first condition rules out at most $\binom{\ell-1}{2}+\ell-2$ points (as we must not pick points that form a block with two of the previous $\ell-1$ points in the partial sequencing, and we must also avoid the points that lie in $\{w_1,w_2,\ldots,w_{i-1}\}\cup\{w\}$). Maintaining the second condition rules out up to $\ell-2$ further points. So when $|W|\geq\binom{\ell-1}{2}+2\ell$, the greedy algorithm always succeeds. This establishes our claim.

\begin{theorem}
\label{thm:general}
Let $\ell$ be a positive integer, and let $V$ be a finite set of order~$n$.
Using the notation above, suppose that $\cF$ has the $L$-suffix, $L'$-prefix, $K$-insertion and $(J,s)$-reachability properties. Then $\cF$ has an $(\ell,\cF)$-good sequencing when
\begin{equation}
\label{eqn:general}
n\geq \ell L+(L+L'+K)L+sL+J.
\end{equation}
If $\cF$ is symmetric, then $\cF$ has an $(\ell,\cF)$-good sequencing when
\begin{equation}
\label{eqn:symmetric}
n\geq \ell L+(L+K)L+sL+J.
\end{equation}
\end{theorem}

Before proving the theorem, we illustrate its usefulness by providing some corollaries. 

\begin{corollary}
\label{cor:STS}
There exists a function $f_{\mathrm{STS}}:\mathbb{N}\rightarrow\mathbb{N}$ such that
any partial Steiner triple tystem of order $n$ with $n\geq f_{\mathrm{STS}}(\ell)$ has an $\ell$-good sequencing. Moreover, 
\[
f_{\mathrm{STS}}(\ell)\leq \tfrac{1}{2}\ell^4+\oh(\ell^3).
\]
\end{corollary}
\begin{proof}
Let $\cF$ be the set of sequences $x_1,x_2,x_3$ such that $\{ x_1,x_2,x_3\}$ is a block in our partial Steiner triple system. Then $\cF$ is symmetric, and has the $L$-suffix, $K$-insertion and $(J,s)$-reachability properties where $L=K=\frac{1}{2}\ell^2+\oh(\ell)$, where $J= \frac{1}{2}\ell^2+\oh(\ell)$ and where $s=\ell+\oh(1)$. The corollary follows by Theorem~\ref{thm:general}.
\end{proof}

Two ordered variations of partial Steiner triple systems have been considered in the context of sequencings: directed triple systems~\cite{KreherStinsonVeitch_DTS} and Mendelsohn triple systems~\cite{KreherStinsonVeitch_Mendelsohn}. We take each in turn, and derive a corollary of Theorem~\ref{thm:general} in this context.

Let $V$ be a finite set. A \emph{transitive triple} is a sequence $x,y,z$ over $V$ where $x$, $y$ and $z$ are distinct. (The terminology comes from the transitive ordering $x<y<z$ determined by the sequence.) We may depict a transitive triple $x,y,z$ as a directed triangle of the following form:
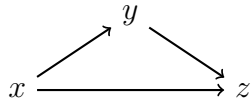
\[
\begin{tikzpicture}
\node (x) at (0,0) {$x$};
\node (y) at (1.5,1) {$y$};
\node (z) at (3,0) {$z$};
\draw [->, thick] (x) to (y);
\draw [->, thick] (x) to (z);
\draw [->, thick] (y) to (z);
\end{tikzpicture}
\]
So a transitive triple can be thought of as a triangle of three directed edges in the complete directed graph on $V$. We define a \emph{partial directed triple system} (DTS) to be a set $\cF$ of transitive triples such that each directed edge in the complete directed graph on $V$ is contained in at most one transitive triple in $\cF$. Following~\cite{KreherStinsonVeitch_DTS},  we define, for an integer $\ell$, a sequencing of a partial directed triple system $\cF$ to be \emph{$\ell$-good} if it is $(\ell,\cF)$-good. So no segment of the sequencing of length at most $\ell$ contains a transitive triple $x,y,z$ from the system as a subsequence. We emphasise that the order of points in each triple is important here. Indeed, if $x,y,z$ is a transitive triple from the system, then $z,x,y$ (for example) is allowed to occur as a subsequence of a segment of an $\ell$-good sequencing.

\begin{corollary}
\label{cor:DTS}
There exists a function $f_{\mathrm{DTS}}:\mathbb{N}\rightarrow\mathbb{N}$ such that
any partial directed triple system of order $n$ with $n\geq f_{\mathrm{DTS}}(\ell)$ has an $\ell$-good sequencing. Moreover, 
\[
f_{\mathrm{DTS}}(\ell)\leq \tfrac{3}{4}\ell^4+\oh(\ell^3).
\]
\end{corollary}
\begin{proof}
A partial directed triple system is not necessarily symmetric. But nevertheless it is not difficult to show that $\cF$ has the $\binom{\ell}{2}$-suffix, $\binom{\ell}{2}$-prefix and $\binom{\ell}{2}$-insertion properties, and the $(\binom{\ell-1}{2}+2\ell,\ell-1)$-reachability property. The corollary follows by Theorem~\ref{thm:general}.
\end{proof}

A $3$-cycle $(x,y,z)$ over $V$, where $x$, $y$ and $z$ are distinct,  may be thought of as a directed triangle of the following form:
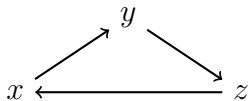
\[
\begin{tikzpicture}
\node (x) at (0,0) {$x$};
\node (y) at (1.5,1) {$y$};
\node (z) at (3,0) {$z$};
\draw [->, thick] (x) to (y);
\draw [<-, thick] (x) to (z);
\draw [->, thick] (y) to (z);
\end{tikzpicture}
\]
Note that $(x,y,z)$, $(y,z,x)$ and $(z,x,y)$ are the same cycle. We may think of a $3$-cycle as a triangle of three directed edges in the complete directed graph on $V$. We define a \emph{partial Mendelsohn triple system} (MTS) to be a set of $3$-cycles over $V$ such that each directed edge in the complete directed graph on $V$ is contained in at most one of these $3$-cycles. A partial Mendelsohn triple system over $V$ gives rise to a set $\cF$ of sequences over $V$ by identifying each $3$-cycle $(x,y,z)$ with three sequences: $x,y,z$, and $y,z,x$, and $z,x,y$. We define, for an integer $\ell$, a sequencing of a partial Mendelsohn triple system $\cF$ to be \emph{$\ell$-good} if it is $(\ell,\cF)$-good. Again, we emphasise that the order of points in each triple is important here: if the $3$-cycle $(x,y,z)$ is in our system, then $x,z,y$ (for example) is allowed to occur as a subsequence of a segment of an $\ell$-good sequencing, though the subsequences $y,z,x$ and $z,x,y$ are forbidden.

Note that the definition of an $\ell$-good partial MTS sequencing we have given differs from~\cite{KreherStinsonVeitch_Mendelsohn}, as there the sequencing itself is regarded as an $n$-cycle. So all segments of length~$\ell$ in the $n$-cycle, including those that overlap the ends of the sequencing, need to be $\cF$-good according to the definition in~\cite{KreherStinsonVeitch_Mendelsohn}. We consider this cyclic version of the $\ell$-good property in Section~\ref{sec:cyclic}.

\begin{corollary}
\label{cor:MTS}
There exists a function $f_{\mathrm{MTS}}:\mathbb{N}\rightarrow\mathbb{N}$ such that
any partial Mendelsohn triple system of order $n$ with $n\geq f_{\mathrm{MTS}}(\ell)$ has an $\ell$-good (non-cyclic) sequencing. Moreover, 
\[
f_{\mathrm{MTS}}(\ell)\leq \ell^4/2+\oh(\ell).
\]
\end{corollary}
\begin{proof}
Let $\cF$ be the set of sequences $x_1,x_2,x_3$ such that $(x_1,x_2,x_3)$ is a $3$-cycle in our partial Mendelsohn triple system. Then $\cF$ is symmetric. Moreover, it is not difficult to show that $\cF$ has the $L$-suffix, $K$-insertion and $(J,s)$-reachability properties where $L=K=\frac{1}{2}\ell^2+\oh(\ell)$, where $J= \frac{1}{2}\ell^2+\oh(\ell)$, and where $s=\ell+\oh(1)$. The corollary follows by Theorem~\ref{thm:general}.
\end{proof}

Before proving Theorem~\ref{thm:general}, we give one more corollary. Given a design $S_\lambda(t,k,n)$ with point set $V$, we may define a set $\cF$ of sequences by setting $x_1,x_2,\ldots,x_k$ to lie in $\cF$ if and only if $\{x_1,x_2,\ldots,x_k\}$ is a block in the design. We say that a sequencing is $\ell$-good if it is $(\ell,\cF)$-good. When $\ell<k$, all sequencings are $\ell$-good. So we may assume that $\ell\geq k$.

\begin{corollary}
\label{cor:BD}
Let $t$, $k$ and $\lambda$ be fixed integers, with $2\leq t<k$ and $\lambda\geq 1$. There exists a function $f_{\mathrm{BD}}:\mathbb{N}\rightarrow\mathbb{N}$ such that
any design $S_\lambda(t,k,n)$ with $n\geq f_{\mathrm{MTS}}(\ell)$ has an $\ell$-good sequencing. Moreover, 
\[
f_{\mathrm{BD}}(\ell)\leq \oh(\ell^{2t}).
\]
\end{corollary}
\begin{proof}
The set $\cF$ of sequences associated with the design is clearly symmetric. We see that $\cF$ has the $L$-suffix property where $L$ is any upper bound on the number of blocks that intersect $\ell-1$ points in a set of size $k-1$, since the remaining points in these blocks are exactly the points we must avoid to extend our sequencing. Counting the number of $t$-subsets of these blocks (in two ways) shows that we may take $L=\lambda \binom{\ell-1}{t}/\binom{k-1}{t}=\oh(\ell^t)$. 

Let $K$ be the maximum number of blocks that intersect a partial $(\ell,\cF)$-good sequencing  $x_1,x_2,\ldots,x_{2(\ell-1)}$ in $k-1$ points, that intersect both the subsets $\{x_1,x_2,\ldots,x_{\ell-1}\}$ and $\{x_\ell,x_{\ell+1},\ldots,x_{2(\ell-1)}\}$ non-trivially, and are contained in a segment of length $\ell-1$ of $x_1,x_2,\ldots,x_{2(\ell-1)}$. We have that $\cF$ has the $K$-insertion property, since the remaining points of each of these blocks include all those those we wish to count. But clearly
\[
K\leq\lambda\binom{2(\ell-1)}{t}=\oh(\ell^t),
\]
since every block we are counting contains a $t$-subset of $\{x_1,x_2,\ldots,x_{2(\ell-1)}\}$, and each such $t$-subset is contained at most $\lambda$ blocks.

It is not hard to see that $\cF$ satisfies the $(J,\ell-1)$-insertion property with $J=L+\ell-1+\lambda\binom{\ell-1}{t-1}=\oh(\ell^{t})$, using a straightforward generalisation of the argument for Steiner triple systems. The corollary now follows, by Theorem~\ref{thm:general}.
\end{proof}

\begin{proof}[Proof of Theorem~\ref{thm:general}.]
Suppose that~\eqref{eqn:general} holds. When $L=0$, the $L$-suffix property implies that the naive greedy algorithm constructs an $(\ell,\cF)$-good sequencing. So we may assume that $L\geq 1$. Similarly, we may assume that $L'\geq 1$. 

We construct a $(\ell,\cF)$-good sequencing as follows.

\paragraph{Stage~1:} We begin by constructing a sequence $x_1,x_2,\ldots,x_{(\ell-1)L}$ that is a partial $(\ell,\cF)$-good sequencing. We do this in a greedy fashion, by choosing each element $x_i\in V$ so that:
\begin{itemize}
\item[(a)] the element $x_i$ is distinct from $x_1,x_2,\ldots, x_{i-1}$, and
\item[(b)] the sequence $x_r,x_{r+1},\ldots,x_i$, where $r=\max\{1,i-(\ell-1)\}$, is $\cF$-good.
\end{itemize}
When choosing $x_i$, at most $i-1$ elements from $V$ are ruled out because of condition~(a) and, since $\cF$ has the $L$-suffix property, at most $L$ choices for $x_i$ are ruled out by condition~(b). Since $n\geq (\ell-1) L+L> i-1+L$, there are always choices for the elements $x_i$ and so the greedy algorithm will always succeed in constructing this partial sequencing.

We divide the partial sequencing into $L$ segments $\bx^1,\bx^2,\ldots ,\bx^L$, each of length $\ell-1$. Define $V'=V\setminus \{x_1,x_2,\ldots,x_{(\ell-1)L}\}$. The initial segment of the sequencing we construct will have the form $y_1\bx^1y_2\bx^2\cdots y_L\bx^L$ for elements $y_1,y_2,\ldots,y_L\in V'$ that are yet to be specified. Not all choices of elements $y_j\in V'$ will preserve the $(\ell,\cF)$-good property of the sequence. We say that $y\in V'$ is \emph{unfortunate} if the $\ell$-good property of the sequencing fails in this way, and let $U\subseteq V'$ be the set of unfortunate elements. So an element $y\in V'$ lies in $U$ exactly when
\begin{itemize}
\item $y\bx^{i} $ fails to be $(\ell,\cF)$-good for some $i\in\{1,2,\ldots, L\}$, or
\item $\bx^{i} y$ fails to be $(\ell,\cF)$-good for some $i\in\{1,2,\ldots ,L-1\}$, or
\item the above two cases do not hold, but $\bx^{i} y\bx^{i+1}$ fails to be $(\ell,\cF)$-good for some $i\in\{1,2,\ldots ,L-1\}$.
\end{itemize}
Since $\cF$ has the $L$-suffix, $L'$-prefix and $K$-insertion properties, and since there are at most $L$ choices for $i$ in the conditions above, we see that $|U|\leq (L+L'+K)L$.

Our strategy will be to extend $\bx^L$ to an $(\ell,\cF)$-good partial sequencing $\bx^L\bz$ whose elements consist of all elements of $U$, and all but $L$ elements $y_1,y_2,\ldots ,y_L$ of $V'$. Since $y_1,y_2,\ldots ,y_L\notin U$, we find that $y_1\bx^1y_2\bx^2\cdots y_L\bx^L$ is a partial $(\ell,\cF)$-good sequencing. Since the sets of elements occurring in the partial sequencings $y_1\bx^1y_2\bx^2\cdots y_L\bx^L$ and $\bz$ are pairwise disjoint, we see that $y_1\bx^1y_2\bx^2\cdots y_L\bx^L\bz$ is a sequencing. Moreover, this sequencing is $(\ell,\cF)$-good, since $\bx^L$ has length $\ell-1$ and so every segment of length $\ell$ or less is a segment of one of the $(\ell,\cF)$-good partial sequencings $\bx^0y_1\bx^1y_2\bx^2\cdots y_L\bx^L$ and $\bx^L\bz$. 

So to establish the theorem, it remains to construct the sequence $\bz$ with the properties we require. We continue as follows.

\paragraph{Stage~2:} If $|U|>L$, we greedily extend the $(\ell,\cF)$-good partial sequencing $\bx^L$ using $|U|-L$ unfortunate elements (to produce another, longer, $(\ell,\cF)$-good partial sequencing). This is possible since at each stage at most $L$ elements are ruled out by the $(\ell,\cF)$-good condition, by the $L$-suffix property, and because there are always more than $L$ unfortunate elements available.

\paragraph{Stage~3:} Next we use the $(J,s)$ reachability property to add each of the remaining unfortunate elements to our partial sequencing (maintaining the $(\ell,\cF)$-good property). So for each remaining element $u\in U$, defining $W\subseteq V'$ to be the set of elements we have not yet used, we extend the partial sequencing by $w_1w_2\cdots w_{s}u$, where $w_1,w_2,\ldots,w_{s}\in W$ are chosen so that they maintain the $(\ell,\cF)$-good property. At any stage of this process, we have used at most $(\ell-1)L+|U|+sL$ elements, and so~\eqref{eqn:general} implies that
\[
|W|\geq n-\big((\ell-1)L+(L+L'+K)L+sL\big)\geq J.
\]
So the $(J,s)$-reachability property guarantees the existence of the elements~$w_i$.

After Stage~3, all unfortunate elements have been used in our partial sequencing. Moreover, since $n\geq  \big((\ell-1)L+(L+L'+K)L+sL+J\big)+L$ at least $L$ elements of $V'$ have not yet been used.

\paragraph{Stage~4:} We now greedily extend our partial $(\ell,\cF)$-good sequencing by elements of $V'$ until just $L$ unused elements remain. The greedy algorithm always succeeds, by the $L$-suffix property. At the end of this process, we have obtained a partial $(\ell,\cF)$-good sequencing $\bx^L\bz$ with the properties we require: $\bz$ contains all but $L$ of the elements of $V'$, and all the elements of~$U$. The first statement of the theorem now follows by the remarks at the end of Stage~1.

\mbox{}

Now suppose in addition that $\cF$ is symmetric, and~\eqref{eqn:symmetric} holds. The strategy above constructs a $(\ell,\cF)$-good sequencing once we note that the number of unfortunate elements is at most $(L+K)L$. This follows since the set of elements $y\in V'$ where $y\bx^i$ is not  $(\ell,\cF)$-good is equal to the set of elements $y\in V'$ where $\bx^iy$ is not  $(\ell,\cF)$-good. 
\end{proof}

\section{Cyclic results}
\label{sec:cyclic}

This section considers a cyclic variant of the problems in Section~\ref{sec:general}. This approach has been considered for Mendelsohn triple systems by Kreher, Stinson and Veitch~\cite{KreherStinsonVeitch_Mendelsohn}, but (to our knowledge) has not been considered previously for Steiner triple systems, directed triple systems or block designs. We formalise this as follows.

The \emph{cyclic segments} of a sequence $x_1,x_2,\ldots,x_n$ are the segments of the sequence, together with the sequences $x_r,x_{r+1},\ldots,x_n,x_1,x_2,\ldots,x_{s}$ for $1\leq s<r\leq n$.

Let $V$ be a set of cardinality $n$, and let $\cF$ be a set of sequences over $V$. We say that a sequencing $x_1,x_2,\ldots,x_n$ is a \emph{cyclic $(\ell,\cF)$-good sequencing} if all its cyclic segments of length at most $\ell$ are $\cF$-good. (When we consider Mendelsohn triple systems, we remark that a cyclic $(\ell,\cF)$-good sequencing is exactly what Kreher et al.~\cite{KreherStinsonVeitch_DTS} call an $\ell$-good sequencing.)

Roughly speaking, our strategy for proving a cyclic variation of Theorem~\ref{thm:general} is to produce a long $(\ell,\cF)$-good partial sequencing with `gaps' as before, then to add elements to produce a cyclic $(\ell,\cF)$-good sequencing with gaps, before finally filling the gaps. 
We need one more definition, closely related to the $K$-insertion property, which is associated with the penultimate step in this process. Let $K'$ and $s'$ be positive integers. We say that $\cF$ has the \emph{$(K',s')$-completion property} if, for any partial $(\ell,\cF)$-good sequencings $\bx$ and $\bx'$ of length $\ell-1$ over $V$, and any subset $X\subseteq V$ of cardinality at least $K'$ and disjoint from the elements in $\bx$ and $\bx'$, there exists a sequence $\by$ of length $s'$ over $X$ such that $\bx\by\bx'$ is a partial $(\ell,\cF)$-good sequencing.

\begin{theorem}
\label{thm:cyclic}
Let $\ell$ be a positive integer, and let $V$ be a finite set of order $n$.
Using the notation of Section~\ref{sec:general} and the notation above, suppose that $\cF$ has the $L$-suffix, $L'$-prefix, $K$-insertion, $(J,s)$-reachability and $(K',s')$-completion properties. Then $\cF$ has a cyclic $(\ell,\cF)$-good sequencing when
\begin{equation}
\label{eqn:general_cyclic}
n\geq (K'-s')\ell+\ell+(K'-s')(L+L'+K)+(s-1)L+J+K'.
\end{equation}
If $\cF$ is symmetric, then $\cF$ has a cyclic $(\ell,\cF)$-good sequencing when
\begin{equation}
\label{eqn:symmetric_cyclic}
n\geq (K'-s')\ell+\ell+(K'-s')(L+K)+(s-1)L+J+K'.
\end{equation}
\end{theorem}
\begin{proof}
The algorithm in Theorem~\ref{thm:general} can be modified to construct an $(\ell,\cF)$-good cycle, as follows.

At \textbf{Stage~1}, we greedily choose an $(\ell,\cF)$-good sequencing of the form $\bx^0\bx^1\cdots \bx^{K'-s'}$, where each sequence $\bx^i$ has length $\ell-1$. The cycle we construct will start with the sequence
$\bx^0y_1\bx^1y_2\bx^2\cdots y_{K'-s'}\bx^{K'-s'}$ where the elements $y_i$ are yet to be specified. We define $V'$ to be the set of elements of $V$ that do not occur in any of the sequences $\bx^i$, and define the set $U$ of unfortunate elements to be those elements $u\in U$ such that $\bx^iu\bx^{i+1}$ fails to be $(\ell,\cF)$-good for some $i$. Counting as before, we see that
\[
|U|\leq (K'-s')(L+L'+K),
\]
since $\cF$ satisfies the $L$-suffix, $L'$-prefix and $K$-insertion properties. When $\cF$ is symmetric,
we have the following slightly stronger upper bound for $|U|$:
\[
|U|\leq (K'-s')(L+K).
\]

We now extend $\bx^{K'-s'}$ to a longer partial $(s,\cF)$-good sequencing including all unfortunate elements, just as before. So in \textbf{Stage~2}, we greedily extend $\bx^{K'-s'}$ using elements of $U$ until at most $L$ unfortunate elements remain, and in \textbf{Stage~3} we use the $(J,s)$-reachability property to add the remaining $L$ unfortunate elements to our partial $(s,\cF)$-good sequencing. At the end of this process, we have used at most $|U|+(s-1)L$ elements from $V'$.

We claim that, provided $L$ is taken to be as small as possible subject to $\cF$ having the $L$-suffix property, we must have that $K'\geq L$. To see this, we note that the minimality of $L$ implies that there must exist a partial $(\ell,\cF)$-good sequencing $\bx$ of length at most $\ell-1$ and a disjoint subset $W\subseteq V$ with $|W|=L-1$ such that none of the sequences $\bx z$ with $z\in W$ are  $(\ell,\cF)$-good. Then $K'>L-1$, since otherwise $\bx$ and $W$ (together with another sequencing $\bx'$) would form a counterexample to the $(K',s)$-completion property. This establishes our claim.

In \textbf{Stage~4}, we continue to greedily add elements to our partial $(s,\cF)$-good sequencing until just $K'$ elements of $V'$ remain unused. Note that this makes sense, since at least $K'$ elements of $V'$ remain unused at the start of this process:
\[
|V'|-(|U|+(s-1)L)\geq n-(K'-s')L-(K'-s')(L+L'+K)-(s-1)L\geq K'
\]
in general, and
\[
|V'|-(|U|+(s-1)L)\geq n-(K'-s')L-(K'-s')(L+K)-(s-1)L\geq K'
\]
when $\cF$ is symmetric. Also note that the greedy algorithm succeeds, since we may assume that $K'\geq L$. At the end of this process we have a sequence $\bz$ of length $|V'|-K'$ over $V'$ such that $\bx^{K'-s'}\bz$ is a partial $(s,\cF)$-good sequencing. Let $Y$ be the set of elements of $V'$ that do not occur in $\bz$.

Define $\bx$ to be the sequence consisting of the final $\ell-1$ entries of $\bz$.
Since $|Y|= K'$, the $(K',s')$ completion property implies there exists a sequence $\by$ over $Y$ of length $s'$ such that $\bx\by\bx^0$ is a partial $(\ell,\cF)$-good sequencing. Let $y_1,y_2,\ldots,y_{K'-s'}$ be the elements of $Y$ that do not occur in $\by$. In \textbf{Stage~5} we output the cycle
\[
\bx^0y_1\bx^1y_2\bx^2\cdots y_{K'-s'}\bx^{K'-s'}\bz\by.
\]
Note that $\bx^0y_1\bx^1y_2\bx^2\cdots y_{K'-s'}\bx^{K'-s'}$  is $(\ell,\cF)$-good, as none of the elements $y_i$ are unfortunate. Moreover, the sequences $\bx^{K'-s'}\bz$ and $\bx\by\bx^0$ are also $(\ell,\cF)$-good by our construction of $\bz$ and by the $(K',s')$-completion property. So we have a cyclic $(\ell,\cF)$-sequencing, as required.
\end{proof}

For a partial STS, MTS, DTS or block design, a \emph{cyclic $\ell$-good sequencing} is a cyclic $(\ell,\cF)$-good sequencing, where $\cF$ is the set of sequences defined in Section~\ref{sec:general}.

\begin{corollary}
\label{cor:STS_cyclic}
There exists a function $g_{\mathrm{STS}}:\mathbb{N}\rightarrow\mathbb{N}$ such that
any partial Steiner triple system of order $n$ with $n\geq g_{\mathrm{STS}}(\ell)$ has a cyclic $\ell$-good sequencing. Moreover, 
\[
g_{\mathrm{STS}}(\ell)\leq\tfrac{3}{2}\ell^4+\oh(\ell^3).
\]
\end{corollary}
\begin{proof}
We fix a partial Steiner triple system, and define $\cF$ to be the set of sequences of length $3$ whose elements form the points in some block. Recall that $\cF$ is symmetric, and has the $L$-suffix, $L'$-prefix, $K$-insertion and $(J,\ell-1)$-reachability properties, where $L$, $L'$, $K$ and $J$ are each of the form $\frac{1}{2}\ell^2+\oh(\ell)$. The corollary follows by Theorem~\ref{thm:cyclic}, provided we can show that $\cF$ has the $(K',\ell-1)$ completion property where $K'=\frac{3}{2}\ell^2+\oh(\ell)$. In fact, we will show that $\cF$ has the $(K',\ell-1)$-completion property where $K'=3\binom{\ell-1}{2}+\ell-1$.

Fix $\cF$-good sequences $\bx$ and $\bx'$, and a set $X$ with $|X|\geq K'$ whose elements do not occur in $\bx$ or $\bx'$. We have $\bx=x_1,x_2,\ldots,x_{\ell-1}$ and $\bx'=x'_1,x'_2,\ldots,x'_{\ell-1}$, say, where for any $i$ we have $x_i\notin X$ and $x'_i\notin X$. We may greedily choose elements $y_1,y_2,\ldots,y_{\ell-1}\in X$, choosing $y_i$ so that the following three conditions all hold:
\begin{itemize}
\item[(a)] $y_i\in X\setminus\{y_1,y_2,\dots,y_{i-1}\}$;
\item[(b)] There is no block of the form $\{y_i,a,b\}$ where
\[
a,b\in\{x_i,x_{i+1},\ldots,x_{\ell-1},y_1,y_2,\ldots,y_{i-1}\};
\]
\item[(c)] For $j\in\{1,2,\ldots, i-1\}$, there is no block of the form $\{y_j,y_i,b\}$ where $b\in\{x'_1,x'_2,\ldots,x'_{j}\}$.
\item[(d)] There is no block of the form $\{y_i,a,b\}$ where
\[
a,b\in\{x'_1,x'_{2},\ldots,x'_{i}\}.
\]
\end{itemize}
Note that condition~(a) rules out at most $\ell-2$ elements $y_i\in X$, and condition~(b) rules out at most $\binom{\ell-1}{2}$ elements. Condition~(c) rules out at most $\sum_{j=1}^{i-1}j=\binom{i}{2}\leq \binom{\ell-1}{2}$ elements, and condition~(d) rules out at most $\binom{i}{2}\leq \binom{\ell-1}{2}$ elements. So the greedy algorithm succeeds in finding suitable elements $y_i$, since $|X|\geq K'=3\binom{\ell-1}{2}+\ell-1$. Now consider the sequence
\[
x_1,x_{2},\ldots,x_{\ell-1},y_1,y_2,\ldots,y_{\ell-1},x'_1,x'_2,\ldots,x'_{\ell-1}.
\]
We have a partial sequencing, since condition~(a) guarantees that the elements $y_i$ are distinct. To prove our claim, it suffices to show that this sequence is a partial $(\ell,\cF)$-good sequencing. Suppose for a contradiction that there are is a segment of length $\ell$ containing a block $B$. Since $\bx$ and $\bx'$ are separated by $\ell-1$ elements, $B$ cannot contain both elements of $\bx$ and $\bx'$. Since $\bx$ and $\bx'$ are $(\ell,\cF)$-good, $B$ must contain at least one element of $\{y_1,y_2,\ldots,y_{\ell-1}\}$. Let $i$ be the largest integer so that $y_i\in B$. Condition~(b) shows that $B$ must contain an element from $\bx'$. If $B$ contains just one element from $\bx'$, condition~(c) shows we have a contradiction; if $B$ contains two elements from $\bx'$, this contradicts condition~(d). So the sequence is $(\ell,\cF)$-good, as required.
\end{proof}

The methods of Corollary~\ref{cor:STS} apply equally to the situation when we are interested in Directed or Mendelsohn triple systems to show that we have the $(3\binom{\ell-1}{2}+\ell-1,\ell-1)$-completion property. Theorem~\ref{thm:cyclic}, together with properties from the previous section, then imply the following two corollaries.

\begin{corollary}
\label{cor:DTS_cyclic}
There exists a function $g_{\mathrm{DTS}}:\mathbb{N}\rightarrow\mathbb{N}$ such that
any partial directed triple system of order $n$ with $n\geq g_{\mathrm{DTS}}(\ell)$ has a cyclic $\ell$-good sequencing. Moreover, 
\[
g_{\mathrm{DTS}}(\ell)\leq \tfrac{9}{4}\ell^4+\oh(\ell^3).
\]
\end{corollary}

\begin{corollary}
\label{cor:MTS_cyclic}
There exists a function $g_{\mathrm{MTS}}:\mathbb{N}\rightarrow\mathbb{N}$ such that
any partial Mendelsohn triple system of order $n$ with $n\geq g_{\mathrm{MTS}}(\ell)$ has a cyclic $\ell$-good sequencing. Moreover, 
\[
g_{\mathrm{MTS}}(\ell)\leq \tfrac{3}{2}\ell^4+\oh(\ell^3).
\]
\end{corollary}

Finally, the following corollary holds for block designs:

\begin{corollary}
\label{cor:BD_cyclic}
Let $t$, $k$ and $\lambda$ be fixed integers, with $2\leq t<k$ and $\lambda\geq 1$. There exists a function $g_{\mathrm{BD}}:\mathbb{N}\rightarrow\mathbb{N}$ such that
any design $S_\lambda(t,k,n)$ with $n\geq f_{\mathrm{MTS}}(\ell)$ has a cyclic $\ell$-good sequencing. Moreover, 
\[
g_{\mathrm{BD}}(\ell)\leq \oh(\ell^{2t}).
\]
\end{corollary}
\begin{proof}
Let $\cF$ be the set of all sequences of length $k$ whose points form a block in our design. 
Just as in the proof of Corollary~\ref{cor:STS_cyclic}, the corollary follows from Theorem~\ref{thm:cyclic} and the results of Section~\ref{sec:general} provided we can show that $\cF$ has the $(K',\ell-1)$-completion property where $K'=\oh(\ell^t)$. In particular, it suffices to prove that $\cF$ has the $(K',\ell-1)$-completion property with $K'=\lambda \binom{2(\ell-1)}{t}+\ell-1$.

As in the proof of Corollary~\ref{cor:STS_cyclic}, we fix $\cF$-good sequences $\bx$ and $\bx'$ of length $\ell-1$ and a set $X$ with $|X|\geq K'$ whose elements do not occur in $\bx$ or $\bx'$. We have $\bx=x_1,x_2,\ldots,x_{\ell-1}$ and $\bx'=x'_1,x'_2,\ldots,x'_{\ell-1}$, say, where for any $i$ we have $x_i\notin X$ and $x'_i\notin X$. We greedily choose elements $y_1,y_2,\ldots,y_{\ell-1}\in X$, choosing $y_i$ so that the following two conditions hold:
\begin{itemize}
\item[(a)] $y_i\in X\setminus\{y_1,y_2,\dots,y_{i-1}\}$;
\item[(b)] There is no block that contains $y_i$ and intersects the set $W_i$ in a set of size $k-1$, where
\[
W_i=\{x_i,x_{i+1},\ldots,x_{\ell-1}\}\cup\{y_1,y_2,\ldots,y_{i-1}\}\cup\{x'_1,x'_2,\ldots,x'_{i}\}.
\]
\end{itemize}
Note that enforcing~(a) rules out at most $\ell-2$ choices for $y_i$. Since the set $W_i$ in condition~(b) has cardinality at most $2(\ell-1)$, there are at most $\lambda\binom{2(\ell-1)}{t}$ blocks that intersect $W_i$ in $t$ or more points, and so there are at most $\lambda\binom{3(\ell-1)}{t}$ blocks that intersect $W_i$ in $k-1$ points. Each such block rules out at most one choice for $y_i$, and so enforcing~(b) rules out at most $\lambda\binom{2(\ell-1)}{t}$ choices for~$y_i$. Since $K'>(\ell-2)+\lambda\binom{2(\ell-1)}{t}$, the greedy algorithm always succeeds.

We note that condition~(a) guarantees that
\[
x_1,x_2,\ldots,x_{\ell-1},y_1,y_2,\ldots,y_{\ell-1},x'_1,x'_2,\ldots,x'_{\ell-1}
\]
is a partial sequencing. Moreover, we claim that condition~(b) guarantees that this partial sequencing is $(\ell,\cF)$-good. For suppose that $B$ is a block that is contained in a segment of length $\ell$ in the partial sequencing. Since $\bx$ and $\bx'$ are $\cF$-good, and are separated by $\ell-1$ elements in the sequence, $B$ must intersect $\{y_1,y_2,\ldots,y_{\ell-1}\}$ non-trivially. Let $i$ be the largest integer such that $y_i\in B$. Then $B\subseteq W_i\cup\{y_i\}$, violating condition~(b). This contradiction establishes our claim. So $\cF$ has the $(K',\ell-1)$-completion property with $K'=\lambda \binom{2(\ell-1)}{t}+\ell-1$, as required.
\end{proof}

\section{$\ell$-good STS sequencings with $\ell$ large}
\label{sec:STS}

Stinson and Veitch~\cite[Theorem~2.1]{StinsonVeitch_STS_sequencing} have shown that when an STS on $n$ points has an $\ell$-good sequencing, then $\ell\leq (n+2)/3$. In this section, we establish two results in this spirit. First (Theorem~\ref{thm:sts_large_ell}) we show that the Stinson--Veitch bound is reasonable, by exhibiting an infinite family of Steiner triple systems with $(n+3)/4$-good sequencings. In fact, the sequencings we exhibit are cyclic $(n+3)/4$-good. Secondly (Theorem~\ref{thm:better_SV}), we provide an upper bound for $\ell$ in the cyclic case which is slightly stronger than that implied by the Stinson--Veitch bound.

\subsection{A construction}

\begin{theorem}
\label{thm:sts_large_ell}
Let $m$ be a positive integer with $m\equiv 2\bmod 4$. 
There exists a Steiner triple on $n=6m+1$ points with a cyclic $(n+3)/4$-good sequencing.
\end{theorem}
\begin{proof}
Recall that a \emph{hooked Skolem sequence} of order $m$ is a partition of the set $\{1,2,\ldots 2m-1\}\cup\{2m+1\}$ of integers into $m$ pairs $(a_i,b_i)$, with the property that $b_i-a_i=i$. O'Keefe~\cite{OKeefe} exhibited the hooked Skolem sequence given in Figure~\ref{fig:skolem} in the case when $m=4k+2$ for some integer $k$. Skolem~\cite{Skolem_STS} constructed a Steiner triple system with vertex set $\mathbb{Z}_{6m+1}$ from any hooked Skolem sequence: the blocks of the triple system are those of the form $\{x,x+i,x+m+b_i\}$ where $x\in\mathbb{Z}_{6m+1}$ and where $i\in\{1,2,\ldots m\}$.
\begin{figure}
\[
\begin{array}{cl}
(r,4k+2-r),&\text{ for } r=1,2,\ldots,2k\\
(2k+1,6k+2)\\
(4k+2,6k+3)\\
(4k+3,8k+5)\\
(4k+3+r,8k+4-r),&\text{ for }r=1,2,\ldots,k-1\\
(5k+2+r,7k+3-r),&\text{ for }r=1,2,\ldots,k-1\\
(7k+3,7k+4)
\end{array}
\]
\caption{The pairs $(a,b)$ in the O'Keefe hooked Skolem sequence of order $m$, where $m=4k+2$.}
\label{fig:skolem}
\end{figure}

We claim that the natural sequencing $0,1,\ldots,6m$ of the elements of $\mathbb{Z}_{6m+1}$ is $(m+2k+2)$-good. Since $m+2k+2=3m/2+1=(n+3)/4$, this claim is sufficient to establish the theorem.

To establish our claim, we show that a block $\{x,x+i,x+m+b_i\}$ is not contained in any cyclic segment of our sequencing of length $m+2k+2$. Without loss of generality, we may assume that $x=0$. The shortest non-cyclic segment containing $\{0,i,m+b_i\}$ has length $m+b_i+1$. Examining Figure~\ref{fig:skolem}, we see that $b_i\geq 2k+2$ in all cases. Since $m+b_i+1>m+2k+2$, no non-cyclic segment of length $m+2k+2$ contains our block. Thus any segment of length $m+2k+2$ must `wrap around'. But the shortest such segment containing $0$ and $m+b_i$ has length at least $6m-(m+b_i)+2=5m+2-b_i$. Since our segment must also contain $i$, it must have length at least $5m+3-b_i$. But $b_i\leq 8k+5=2m+1$, so the segment has length at least $3m+2$. Since $3m+2>m+2k+2$, our claim follows.
\end{proof}

\subsection{An upper bound for $\ell$ in the cyclic case}

A cyclic $\ell$-good sequencing is $\ell$-good, so the bound due to Stinson and Veitch~\cite[Theorem~2.1]{StinsonVeitch_STS_sequencing} implies that whenever we have a Steiner triple system on $n$-points with a cyclic $\ell$-good sequencing then $\ell\leq (1/3)n+\oh(1)$. The following theorem shows that this bound is not tight for cyclic sequencings:

\begin{theorem}
\label{thm:better_SV}
If there exists a Steiner triple system on $n$ points with a cyclic $\ell$-good sequencing then
\[
\ell\leq 0.329\, n+\oh(1).
\]
\end{theorem}

Before we present a formal proof, here is a sketch of the basic idea. Let $x_1,x_2,\ldots,x_n$ be a cyclic $\ell$-good sequencing. Since the sequencing is $\ell$-good, the Stinson--Veitch approach can be applied: First count the number of blocks intersect $\{x_1,x_2,\dots x_\ell\}$ trivially. Then provide an upper bound for the number of these blocks by counting pairs $x,y\in\{x_{\ell+1},x_{\ell+2},\ldots,x_n\}$ that are more than $\ell$ positions apart in the sequencing (noting that every block we are counting contains at least one such pair). We obtain a contradiction unless $\ell\leq n/3+\oh(1)$. Our aim is to show that, in the cyclic $\ell$-good case, when $\ell$ is just a little less $n/3$ there is a significant number of blocks that contain two such pairs $x,y$. This allows us to improve our upper bound on blocks, which in turn allows us to improve our upper bound on $\ell$.

\begin{proof}[Proof of Theorem~\ref{thm:better_SV}.]
Let $\ell=\lfloor 0.329\, n\rfloor$, where $n$ is sufficiently large. Fix a Steiner system on $n$ points. Suppose, for a contradiction, that there exists a cyclic $\ell$-good sequencing $x_0,x_1,\ldots,x_{n-1}$ for this system (where we take subscripts modulo $n$). 

Define $\delta=\lceil \ell/2 \rceil$ and $\varepsilon=n-6\delta$. So $\delta=0.1645\, n+\oh(1)$ and $\varepsilon= 0.013\, n+\oh(1)$.

For a shift $r\in\bZ_n$, we divide the sequencing $x_r,x_{r+1},\ldots ,x_{r+n-1}$ into seven segments: six long segments of length $\delta$ and one short segment of length $\varepsilon$. For $i\in\{1,2,\ldots,7\}$, let $S_i(r)$ be the set of points occuring in the $i$th segment. Since our sequencing is $\ell$-good, we find that:
\begin{itemize}
\item[(a)] No block can be contained in $S_i(r)\cup S_{i+1}(r)$ for $i\in\{1,2,\ldots,6\}$.
\item[(b)] If a block is contained in $S_6(r)\cup S_7(r)\cup S_1(r)$, it must contain a point in the first
$\varepsilon$ elements of $S_6(r)$ and the last $\varepsilon$ elements of $S_1(r)$ (in the ordering given by the sequencing). Indeed these points must be at least $2\delta+1$ positions apart.
\end{itemize}
For positive integers $k_1,k_2,\ldots,k_6$, let $b_{k_1k_2k_3k_4k_5k_6}(r)$ be the number of blocks $B$ such that $|B\cap S_i(r)|=k_i$ for $i\in\{1,2,\ldots,6\}$. We refer to the integer $b_{k_1k_2k_3k_4k_5k_6}(r)$ as a \emph{count}. Note that $|B\cap S_7(r)|=3-\sum_{i=1}^6 k_i$. All counts are non-negative, and if our count is non-zero we must have $\sum_{i=1}^6k_i\leq 3$.  Condition~(a) above shows that the count is zero unless $0\leq k_i\leq 2$ for all $i$, and moreover
\begin{equation}
\label{eqn:zero}
\begin{split}
b_{210000}(r)&=b_{021000}(r)=b_{002100}(r)=b_{000210}(r)=b_{000021}(r)=0,\\
b_{120000}(r)&=b_{012000}(r)=b_{001200}(r)=b_{000120}(r)=b_{000012}(r)=0,\\
b_{200000}(r)&=b_{100000}(r)=b_{000002}(r)=b_{000001}(r)=0,\\
b_{000000}(r)&=0.
\end{split}
\end{equation}
By Condition~(b), counting pairs of points $(x,y)$ with $x\in S_6(r)$ and $y\in S_1(r)$ at a distance of at least $2\delta+1$ apart, we find that
\begin{equation}
\label{eqn:small_end}
b_{100001}+b_{200001}+b_{100002}\leq \sum_{x=1}^\varepsilon (\varepsilon-x+1)=\frac{1}{2}\varepsilon(\varepsilon+1)=\binom{\varepsilon+1}{2}.
\end{equation}

For every choice of integers $i,j\in\{1,2,\ldots,7\}$ with $i\leq j$, we obtain a linear equation involving our counts by counting triples $(x,y,B)$ where $x\in S_i(r)$, $y\in S_j(r)$ and $B$ is a block containing $x$ and $y$.  
We obtain $28$ equalities in this way. Typical examples (for $(i,j)=(2,3)$, $(2,4)$, $(2,5)$, $(1,1)$ and $(7,7)$ respectively) are:
\begin{align}
\label{eqn:adjacent}
\delta^2&=b_{111000}(r)+b_{011100}(r)+b_{011010}(r)+b_{011001}(r)+b_{011000}(r),\\
\label{eqn:one_apart}
\begin{split}
\delta^2&=b_{110100}(r)+2b_{020100}(r)+b_{011100}(r)+2b_{010200}(r)+b_{010110}(r)\\
&\quad +b_{010101}(r)+b_{010100}(r),
\end{split}\\
\label{eqn:two_apart}
\begin{split}
\delta^2&=b_{110010}(r)+2b_{020010}(r)+b_{011010}(r)+b_{010110}(r)+2b_{010020}(r)\\
&\quad +b_{010011}(r)+b_{010010}(r),
\end{split}\\
\label{eqn:two_in_a_block}
\mbox{$\binom{\delta}{2}$}&=b_{201000}(r)+b_{200100}(r)+b_{200010}(r)+b_{200001},\\
\label{eqn:weight_one_sum}
\mbox{$\binom{\varepsilon}{2}$}&=b_{010000}(r)+b_{001000}(r)+b_{000100}(r)+b_{000010}(r).
\end{align}

We note that~\eqref{eqn:weight_one_sum} shows that our count $b_{k_1k_2k_3k_4k_5k_6}(r)$ is small whenever $\sum_{i=1}^6k_i=1$. In particular,
\[
b_{k_1k_2k_3k_4k_5k_6}(r)\leq \binom{\varepsilon}{2}\text{ when }\sum_{i=1}^6k_i=1.
\]
When $\sum_{i=1}^6k_i=2$, the count is also fairly small. To see this, note that if $k_j\not=0$ then all the blocks we are counting contain an element $x\in S_j(r)$ and an element in $y\in S_7(r)$. There are $\delta\varepsilon$ choices for the pair $(x,y)$ and this choice determines the block. Hence
\begin{equation}
\label{eqn:weight_two}
b_{k_1k_2k_3k_4k_5k_6}(r)\leq \delta\varepsilon\text{ when }\sum_{i=1}^6k_i=2.
\end{equation}
If in addition $k_j=2$, each block we are counting is associated with two pairs $(x,y)$ and so \begin{equation}
\label{eqn:just_one_two}
b_{k_1k_2k_3k_4k_5k_6}(r)\leq \delta\varepsilon/2\text{ when }
k_i=\begin{cases}
2&\text{ for }i=j,\\
0&\text{ otherwise.}
\end{cases}
\end{equation}

Define
\[
a_{k_1k_2k_3k_4k_5k_6}=\frac{1}{n}\sum_{r=0}^{n-1}b_{k_1k_2k_3k_4k_5k_6}(r),
\]
so we are taking an average value of counts over all shifts. All the equations above are linear in our counts, so remain valid when we replace each instance of $b_{k_1k_2k_3k_4k_5k_6}(r)$ by $a_{k_1k_2k_3k_4k_5k_6}$. Since $b_{102000}(r+\delta)$ and $b_{010200}(r)$ count the same blocks, these counts are equal. So $a_{102000}=a_{010200}$. More generally, we have the following equalities between our averages:
\begin{equation}
\label{eqn:shift_equalities}
\begin{split}
a_{102000}&=a_{010200}=a_{001020}=a_{000102},\\
a_{201000}&=a_{020100}=a_{002010}=a_{000201},\\
a_{100200}&=a_{010020}=a_{001002},\\
a_{200100}&=a_{020010}=a_{002001},\text{ and}\\
a_{200010}&=a_{020001}.
\end{split}
\end{equation}
We also see that some averages, though not necessarily equal, are close in value because $\varepsilon$ is small. For example, we claim that $a_{201000}$ and $a_{010002}$ are close. To see this, first note that the only blocks counted by $b_{201000}(r+5\delta)$ but not by $b_{010002}(r)$ contain two points in $S_1(r+5\delta)=S_6(r)$ and one point in $S_3(r+5\delta)\setminus S_2(r)$. Since $|S_3(r+5\delta)\setminus S_2(r)|=\varepsilon$, the number of blocks of this form is at most $\varepsilon \delta/2$. Similarly, the number of blocks counted by $b_{010002}(r)$ but not by $b_{201000}(r+5\delta)$ is at most $\varepsilon \delta/2$. Hence $|b_{201000}(r+5\delta)-b_{010002}(r)|\leq \delta\varepsilon$ and so
\[
a_{010002}-\delta\varepsilon \leq a_{201000}\leq a_{010002}+\delta\varepsilon,
\]
establishing our claim. For similar reasons, the following inequalities hold:
\begin{align*}
a_{020001}-\delta\varepsilon &\leq a_{102000}\leq a_{020001}+\delta\varepsilon,\\
a_{001002}-\delta\varepsilon &\leq a_{200100}\leq a_{001002}+\delta\varepsilon,\\
a_{100110}-2\delta\varepsilon&\leq a_{001101}\leq a_{100110}+2\delta\varepsilon,\text{ and}\\
a_{101100}-2\delta\varepsilon&\leq a_{011001}\leq a_{101100}+2\delta\varepsilon.
\end{align*}

Now equations~\eqref{eqn:adjacent},~\eqref{eqn:weight_two} and~\eqref{eqn:shift_equalities} together imply that
\begin{equation}
\label{eqn:first}
\delta^2-\delta\varepsilon\leq 2a_{111000}+a_{110100}+a_{110010}\leq \delta^2.
\end{equation}
Similarly, equations~\eqref{eqn:one_apart}, \eqref{eqn:two_apart} and~\eqref{eqn:two_in_a_block} respectively become
\begin{align}
\label{eqn:second}
\begin{split}
\delta^2-5\delta\varepsilon&\leq a_{110100}+2a_{201000}+a_{111000}+2a_{200010}+a_{110010}+a_{101010}\\
&\leq \delta^2+4\delta\varepsilon
\end{split}\\
\label{eqn:third}
\delta^2-7\delta\varepsilon&\leq 2a_{110010}+4a_{200100}+2a_{110100}\leq \delta^2+6\delta\varepsilon,\text{ and}\\
\label{eqn:fourth}
\mbox{$\binom{\delta}{2}-\binom{\varepsilon+1}{2}$}&\leq a_{201000}+a_{200100}+a_{200010}\leq \mbox{$\binom{\delta}{2}$}.
\end{align}

Let $t$ be the minimum value of $a_{200100}$ subject to the inequalities~\eqref{eqn:first}, \eqref{eqn:second},\eqref{eqn:third} and~\eqref{eqn:fourth}. A linear programming exercise (we used Mathematica~\cite{Mathematica}) shows that $t=0.00225352\, n^2+\oh(n)$.

Let $r$ be a shift so that $b_{002001}\geq t$. The shift $r$ exists since $a_{200100}=a_{002001}$ and since $a_{002001}$ is an average value of $b_{002001}(r)$ over all shifts. We now derive a contradiction using the method of Stinson and Veitch. The number of blocks intersecting $\{x_r,x_{r+1},\ldots,x_{r+\ell-1}\}$ in two points is $\binom{\ell}{2}$ (no block intersects this set in three points, as the sequencing is $\ell$-good). Since every point is contained in $(n-1)/2$ blocks, the number of blocks intersecting $\{x_r,x_{r+1},\ldots,x_{r+\ell-1}\}$ non-trivially is $\ell(n-1)/2-\binom{\ell}{2}$. Hence the number of blocks intersecting $\{x_r,x_{r+1},\ldots,x_{r+\ell-1}\}$ trivially is $\frac{1}{6}n(n-1)-\ell(n-1)/2+\binom{\ell}{2}$. The number of such blocks is at most the number of pairs of points in $\{x_{\ell},x_{\ell+1},\ldots,z_{n-1}\}$ at a distance $\ell+1$ or more, since every block we are counting contains such a pair of points and since this pair of points determines the block. Indeed, since every block counted by $b_{002001}(r)$ contains two such pairs, we find that
\[
\frac{1}{6}n(n-1)-\ell(n-1)+\binom{\ell}{2}+b_{002001}(r)\leq \sum_{i=\ell}^{n-1-\ell}(n-1-i)=\frac{1}{2}(n-2\ell)(n-2\ell+1).
\]
But it is routine to check that this inequality contradicts the fact that $b_{002001}(r)\geq t$ when $n$ is sufficiently large. So we have a contradiction as required, and the theorem follows.
\end{proof}

The proof of Theorem~\ref{thm:better_SV} does not suffice to show that $\ell\leq 0.328\, n+\oh(1)$, for example, as the inequality in the last paragraph of the proof no longer contradicts the fact that $b_{002001}(r)\geq t$. However, we do not believe the bound on $\ell$ in the theorem is best possible. A slight improvement of the bound might well be produced by modifying the proof more extensively, at the expense of adding more complexity. It would be interesting to see a significantly improved bound, or an improvement over the construction above (in either the cyclic or the general case). 

\section{$\ell$-good sequencings for quadruple systems}
\label{sec:SQS}

In this section, we construct $\ell$-good sequencings for certain families of Steiner quadruple systems, where $\ell$ is large. We also  establish bounds on the maximum value of $\ell$ such that there exists a Steiner quadruple system on $n$ points with an $\ell$-good sequencing.

\subsection{Constructions}

We exhibit a Steiner quadruple system on $n$ points with a cyclic $(n/4+1)$-good sequencing, using a construction which resembles and is inspired by a construction of Lindner~\cite{Lindner}. This construction works only when $n/4$ is even. In fact there is a similar construction that works for infinitely many values of $n$ with $n/4$ odd, using the construction of Etzion and Hartman~\cite{EtHa91}. We note that the constructions in~\cite{Lindner} and~\cite{EtHa91} were introduced to construct pairwise disjoint
Steiner quadruple systems; all these systems have cyclic $(n/4+1)$-good sequencings. We also note that our construction is slightly simpler than those presented in~\cite{Lindner} and~\cite{EtHa91}.

\begin{construction}
\label{con:sqs}
Given a Steiner quadruple system $(Z,\mathcal{B})$ on $m$ points, there exists a Steiner quadruple system on $n=4m$ points with a cyclic $(n/4+1)$-good sequencing.
\end{construction}
\begin{proof}
Let $\mathbb{Z}_m$ be the integers modulo $m$, and fix a Steiner quadruple system $(\mathbb{Z}_m,\mathcal{B})$. So $\mathcal{B}$ is a set of blocks $B$ , each of cardinality $4$, with the property that every three distinct points is contained in a unique block.

Now $m$ is even (indeed, $m \equiv 2 \mod 6$ or $m \equiv 4 \mod 6$), and so there exists a $1$-factorisation $F_1,F_2,\ldots ,F_{m-1}$ of the complete graph on $\mathbb{Z}_m$. We think of the subscripts $1,2,\ldots,m-1$ as the non-zero elements of $\mathbb{Z}_m$.

Let $V=V_1\cup V_2\cup V_3\cup V_4$ be the disjoint union of $4$ copies $V_i$ of $\mathbb{Z}_m$. The blocks of the design are defined in the following table, where the elements in each row are distinct:
\[
\begin{array}{c|c|c|c|c}
V_1&V_2&V_3&V_4&\text{ where}\\\hline
x,y,z&&w&&\{x,y,z,w\}\in\mathcal{B}\\
&x,y,z&&w&\{x,y,z,w\}\in\mathcal{B}\\
w&&x,y,z&&\{x,y,z,w\}\in\mathcal{B}\\
&w&&x,y,z&\{x,y,z,w\}\in\mathcal{B}\\
a,b&&a,b&&\{a,b\}\subseteq \mathbb{Z}_m\\
&a,b&&a,b&\{a,b\}\subseteq \mathbb{Z}_m\\
i&u,v&j&&\{i,j\}\subseteq \mathbb{Z}_m,\, \{u,v\}\in F_{i-j}\\
&i&u,v&j&\{i,j\}\subseteq \mathbb{Z}_m,\, \{u,v\}\in F_{i-j}\\
j&&i&u,v&\{i,j\}\subseteq \mathbb{Z}_m,\, \{u,v\}\in F_{i-j}\\
u,v&j&&i&\{i,j\}\subseteq \mathbb{Z}_m,\, \{u,v\}\in F_{i-j}\\
r&s&r&s&\{r,s\}\subseteq \mathbb{Z}_m\\
r&r&r&r&r\in \mathbb{Z}_m
\end{array}
\]

It is not hard to check that blocks do indeed form a Steiner quadruple system. Now consider the sequencing that lists all the elements of $V_1$, $V_2$, $V_3$ and then $V_4$, with each set $V_i$ listed in numerical order. This is a cyclic $(m+1)$-good sequencing, since no block is contained in the subset $V_i\cup V_{i+1}$ for $i\in\{1,2,3\}$ or the subset $V_4\cup V_1$. 
\end{proof}

\begin{corollary}
\label{cor:SQS_construction}
There exists an infinite family of Steiner quadruple systems on $n$ points with a cyclic $(n/4+1)$-good sequencing.
\end{corollary}
\begin{proof}
The corollary follows by Construction~\ref{con:sqs} and the fact~\cite{Hanani} that there are infinitely many Steiner quadruple systems. 
\end{proof}

\subsection{Upper bounds on $\ell$}

We now establish bounds on the maximum value of $\ell$ such that there exists a Steiner quadruple system on $n$ points with an $\ell$-good sequencing. We first provide an elementary bound for any block design, and then improve this bound for systems $S_\lambda(t,t+1,n)$ where $t$ and $\lambda$ are arbitrary. This latter bound is an extension of the bound of Stinson and Veitch~\cite{StinsonVeitch_STS_sequencing} in the case of Steiner triple systems $S(2,3,n)$.

\begin{theorem}
\label{thm:easy_bound}
If a design $S_\lambda(t,k,n)$ has an $\ell$-good sequencing, then
\[
\binom{n}{t}\leq \binom{k}{t}\sum_{u=1}^{n-\ell-1}\sum_{v=u+\ell+1}^n\binom{v-u-1}{t-2}.
\]
\end{theorem}
\begin{proof}
Let $x_1,x_2,\ldots,x_n$ be an $\ell$-good sequencing. We say that a $t$-subset $X\subseteq \{x_1,x_2,\ldots,x_n\}$ is \emph{wide} if the first element $x_u$ and last element $x_v$ in the sequencing that lie in $X$ are more than $\ell$ positions apart: $v-u>\ell$. All blocks contain a wide $t$-subset, because the sequencing is $\ell$-good. Now, the number of wide $t$-subsets is
\[
\sum_{u=1}^{n-\ell-1}\sum_{v=u+\ell+1}^n\binom{v-u-1}{t-2}.
\]
Moreover, each $t$-subset is contained in $\lambda$ blocks, and there are $\lambda\binom{n}{t}/\binom{k}{t}$ blocks. So the theorem follows by counting pairs $(X,B)$ where $X$ is a wide $t$-subset and $B$ is a block containing $X$.
\end{proof}

In the case of Steiner quadruple systems, Theorem~\ref{thm:easy_bound} and a short calculation shows that
\[
n^3/6\leq 4\left(n^3/6-\ell^2n/2+\ell^3/3\right)+\oh(n^2),
\]
which implies that $\ell\leq 0.67365 n+o(n)$.
However, we obtain a better bound using the following generalisation of the Stinson--Veitch bound (see Corollary~\ref{cor:SV_asymptotic} below):

\begin{theorem}
\label{thm:SV_generalisation}
Let $t$, $k$ and $\lambda$ be positive integers such that $2\leq t$ and $k=t+1$. Consider the solution $b_0,b_1,\ldots,b_{t+1}$ to the system of equations given by $b_{t+1}=0$ and
\begin{equation}
\label{eqn:bi_equations}
b_i=\frac{\left(\lambda\binom{\ell}{i}\binom{n-\ell}{t-i}-(i+1)b_{i+1}\right)}{t+1-i}\text{ for }i=0,1,\ldots,t.
\end{equation}
If there exists an $S_\lambda(t,k,n)$ with a (not necessarily cyclic) $\ell$-good sequencing, then
\begin{equation}
\label{eqn:b0_inequality}
b_0\leq \sum_{u=\ell+1}^{n-\ell}\sum_{v=u+\ell}^{n}\lambda\binom{v-u-1}{t-2}/(k-2).
\end{equation}
\end{theorem}
\begin{proof} This proof generalises that of Stinson and Veitch~\cite{StinsonVeitch_STS_sequencing}, who consider the special case when $t=2$ and $\lambda=1$. 

Assume that there exists an $S_\lambda(t,k,n)$ that has an $\ell$-good sequencing $x_1,x_2,\ldots,x_n$. Let $X=\{x_1,x_2,\ldots,x_{\ell}\}$. Let $b_i$ be the number of blocks that intersect $X$ in $i$ positions. We find linear equations that the numbers $b_i$ satisfy as follows. Clearly $b_i=0$ for $i\geq k+1$, and since our sequencing is $\ell$-good we see that $b_k=0$. Fix $i\in\{0,1,\ldots,t\}$. The number of $t$-subsets $T$ with $|T\cap X|=i$ is $\binom{\ell}{i}\binom{n-\ell}{t-i}$. If $B$ is a block with $|B\cap X|=j$ then $B$ contains $\binom{j}{i}\binom{k-j}{t-i}$ such $t$-subsets. Moreover, each $t$-subset $T$ is contained in exactly $\lambda$ blocks. So counting pairs $(T,B)$ with $T\subseteq B$ in two ways, and recalling that $k=t+1$, we see that
\begin{align*}
\lambda\binom{\ell}{i}\binom{n-\ell}{t-i}&=\sum_{j=i}^{k-1} \binom{j}{i}\binom{k-j}{t-i} b_j\\
&=\binom{k-i}{t-i}b_i+\binom{i+1}{i}\binom{k-i-1}{t-i}b_{i+1}\\
&=(k-i)b_i+(i+1)b_{i+1}.
\end{align*}
So the equations~\eqref{eqn:bi_equations} hold. 

We now establish the upper bound~\eqref{eqn:b0_inequality} on $b_0$. Any block $B$ contains an initial element $x_u$ and a final element $x_v$ from the sequencing, where $v-u>\ell$. If $B$ is a block counted by $b_0$, we must have $u>\ell$. For a fixed choice of $u$ and $v$, the number of $(t-2)$-subsets $X$ of $\{x_{u+1},x_{u+2},\ldots,x_{v-1}\}$ is $\binom{v-u-1}{t-2}$. There are $\lambda$ blocks containing the $t$-set $X\cup\{x_u,x_v\}$ (some of which might not have initial element $x_u$ and a final element $x_v$). Every block with initial element $x_u$ and a final element $x_v$ contains $\binom{k-2}{t-2}$ subsets $X$. Hence
\[
b_0\leq \sum_{u=\ell+1}^{n-\ell}\sum_{v=u+\ell}^{n}\lambda\binom{v-u-1}{t-2}/\binom{k-2}{t-2}.
\]
Since $k=t+1$, we see that $\binom{k-2}{t-2}=k-2$, and so \eqref{eqn:b0_inequality} follows, as required.
\end{proof}

We note that when $t=2$, Theorem~\ref{thm:SV_generalisation} is exactly the Stinson-Veitch bound. For Steiner quadruple systems, Theorem~\ref{thm:SV_generalisation} implies the following asymptotic result:

\begin{corollary}
\label{cor:SV_asymptotic}
Let $\alpha\in (0,1)$ be fixed. Suppose there exists an infinite collection of Steiner quadruple systems SQS$(n)$ that have $\lfloor \alpha n\rfloor$-good sequencings. Then $\alpha\leq 1/\sqrt{6}<0.41$.
\end{corollary}
\begin{proof}
Consider the integers $b_i$ defined in Theorem~\ref{thm:SV_generalisation}. Using the equations~\eqref{eqn:bi_equations} in the case when $t=3$ and $\lambda=1$, and using the fact that
\[
\binom{\lfloor \alpha n\rfloor}{i}\binom{n-\lfloor \alpha n\rfloor}{3-i} = \frac{\alpha^i(1-\alpha)^{3-i}}{i!(3-i)!}n^3+\oh(n^2),
\]
we see that
\begin{align*}
b_0&=\frac{1}{24}(1-\alpha)^3n^3-\frac{1}{4}b_1+\oh(n^2)\\
b_1&=\frac{1}{6}\alpha(1-\alpha)^2n^3-\frac{2}{3}b_2+\oh(n^2)\\
b_2&=\frac{1}{4}\alpha^2(1-\alpha)n^3-\frac{3}{2}b_3+\oh(n^3)\text{ and}\\
b_3&=\frac{1}{6}\alpha^3-4b_4+\oh(n^3)
\end{align*}
and $b_4=0$, and so
\begin{equation}
\label{eqn:b0_SQS}
b_0=\frac{1}{24}(1-4\alpha+6\alpha^2-4\alpha^3)n^3+\oh(n^2).
\end{equation}
Moreover, the bound~\eqref{eqn:b0_inequality} may be written as \begin{equation}
\label{eqn:b0_bound_SQS}
b_0\leq \frac{1}{2}\left(\frac{1}{6}-\frac{\alpha}{2}+\frac{2\alpha^3}{3}\right)n^3+\oh(n^2).
\end{equation}
To see this, note that the right hand side of~\eqref{eqn:b0_inequality} is
\[
\sum_{u=\ell+1}^{n-\ell}\sum_{v=u+\ell}^{n}\tfrac{1}{2}(v-u-1)=\tfrac{1}{2}\beta n^3+\oh(n^3),
\]
where
\[
\beta=\int_\alpha^{1-\alpha}\left(\int_{x+\alpha}^1(y-x) \,dy\right)dx=\frac{1}{6}-\frac{\alpha}{2}+\frac{2\alpha^3}{3}.
\]
Combining~\eqref{eqn:b0_SQS} and~\eqref{eqn:b0_bound_SQS}, we find that
\[
0\leq 1-2\alpha-6\alpha^2+12\alpha^3.
\]
We note that for any number $x$ such that $0\leq x\leq \alpha$, the collection of Steiner quadruple systems SQS($n$) have $\lfloor x n\rfloor$-good sequencings. So in fact
\[
0\leq 1-2x-6x^2+12x^3 \text{ for all }x\in [0,\alpha].
\]
This implies that $\alpha$ can be no larger than the smallest positive root $1/\sqrt{6}$ of the polynomial $1-2x-6x^2+12x^3$, as required.
\end{proof}

\section{Sequenceable partial Steiner triple systems}
\label{sec:alspach}

Recall~\cite{Alspach,AlspachKreherPastine} that a partial Steiner triple system is \emph{sequenceable} if there exists a sequencing so that for all $r$, no segment of length $3r$ consists of the points of $r$ pairwise disjoint blocks.

\begin{theorem}
\label{thm:alspach}
Fix a partial Steiner triple system on a point set $V$ of cardinality $n$. Suppose that the triple system has $k$ disjoint blocks, but does not have $k+1$ disjoint blocks. Then the triple system is sequenceable provided that
\begin{equation}
\label{eqn:alspach}
n>9k+22k^{2/3}+10.
\end{equation}
\end{theorem}
Let $X$ be the set of points in some union of $k$ disjoint blocks. A key observation due to Alspach \emph{et al.}~\cite{AlspachKreherPastine} is that any segment of length $3r$ that contains fewer than $r$ elements of $X$ cannot be the union of disjoint blocks. For then one of the $r$ blocks will be disjoint from $X$, which would mean that the system contains $k+1$ disjoint blocks, contradicting the definition of $k$. Using this observation, Alspach \emph{et al.} showed that a triple system is sequenceable provided $n\geq 15k-5$. We use an extra trick to improve the bound on $n$: we use the key observation for long segments, but use a greedy algorithm to make sure that short segments are also not unions of disjoint blocks.

\begin{proof}[Proof of Theorem~\ref{thm:alspach}.]
Let $X$ be the set of points in the union of a set of $k$ disjoint blocks, so $|X|=3k$. Define $Y=V\setminus X$, so $Y$ is the set of points not in the union of these $k$ disjoint blocks. Suppose that the lower bound~\eqref{eqn:alspach} on $n$ holds. To prove the theorem, it suffices to construct a sequencing of the form we want.

We begin by determining where the elements of $X$ and $Y$ should be placed in our sequencing. We construct a binary sequence $\bs$, where the positions of zeroes in $\bs$ will determine positions of elements of $X$ in our sequencing, as follows. Let $\ell=\lfloor k^{1/3}\rfloor$. We begin by concatenating $\lfloor 3k/\ell\rfloor$ copies of sequence $(011)^{\ell-1}0111$ of length $3\ell+1$. If necessary, we append copies of the sequence $011$ until the symbol zero occurs $3k$ times in the sequence. We now have a sequence of length $9k+\lfloor 3k/\ell\rfloor$ with zero occurring $3k$ times; these occurrences are separated by runs of ones of length $2$ or $3$. Every $\ell$-th run of ones has length~$3$, and the remainder have length $2$. Finally, we append copies of the symbol one until the sequence has length $n$.  We claim that the sequence $\bs$ that results has the following three properties:
\begin{itemize}
\item[(a)] Every segment of $\bs$ of length $3r$ contains at most $r$ zero entries.
\item[(b)] If a segment of $\bs$ of length $3r$ ends at position $9k+\lfloor 3k/\ell\rfloor+1$ or later, the segment contains at most $r-1$ zero entries.
\item[(c)] When $r\geq 3\ell+1$, a segment of $\bs$ of length $3r$ contains at most $r-1$ zero entries.
\end{itemize}
Claim (a) follows since zero entries are separated by $11$ or $111$, so no three consecutive entries can contain more than one zero. Claim~(b) follows since the first $3(r-1)$ entries of the segment contain at most $r-1$ zero entries, by Claim (a), and the final 3 entries of the segment are $111$. To see why Claim~(c) is true, suppose for a contradiction that the segment contains $r$ zeroes. The segment must therefore contain at least $r-1$ runs of elements equal to one. All these runs have at least length $2$, and since $r-1\geq 3\ell$ at least three of these runs have length $3$ (since every $\ell$th run of ones has length $3$). The segment will therefore contain $r$ occurrences of zero and at least $2(r-1)+3$ occurrences of one. But this is impossible as the segment has length $3r$. 

We now greedily construct our sequencing $z_1,z_2,\ldots,z_n$ as follows. Let $X=\{x_1,x_2,\ldots,x_{3k}\}$. For each $i\in\{1,2,\ldots,n\}$ in turn, let $j$ be the number of zeroes in the initial segment $s_1,s_2,\ldots ,s_i$ of $\bs$. If the $i$th entry $s_i$ of $\bs$ is equal to $0$, we set $z_i=x_j\in X$. If $s_i=1$ and $i>9k+\lfloor 3k/\ell\rfloor$ we set $z_i$ to be any unused element of $Y$. Otherwise we set $z_i=y\in Y$, where $y$ is chosen as follows. We define the set $X'\subseteq X$ by
\[
X'=\{x_m: \min(1,j-3\ell+1)\leq m\leq \min(3k,j+1)\},
\]
so $X'$ is made up of the last $3\ell$ elements of $X$ we have used, together with the next element of $X$ (if any) we will use. We set $Y'$ to be the last $6\ell$ elements of $Y'$ we have used (or all used elements of $Y$ if we have not yet used $6\ell$ elements). We choose $y\in Y$ so that there is no block of the form $\{x',y',y\}$ for $x\in X'$ and $y'\in Y'$. Note that this last condition rules out at most $|X'||Y'|\leq 18\ell^2+6\ell$ elements. Since $i\leq 9k+\lfloor 3k/\ell\rfloor$, the number of choices for $Y$ is at least
\[
n-(9k+\lfloor 3k/\ell\rfloor)-18\ell^2-6\ell>0
\]
by~\eqref{eqn:alspach}. Thus the algorithm always succeeds.

Finally, we check that the sequencing $z_1,z_2,\ldots,z_n$ has the property we want. Let $r$ be a positive integer, and consider a segment of the sequencing of length $3r$. Suppose for a contradiction that the segment can be expressed as the union of $r$ disjoint blocks $B_1,B_2,\ldots,B_r$. The segment contains at most $r$ elements of $X$, by condition~(a) above, and no block $B_i$ can be disjoint from $X$, since $X$ is the set of points in a maximal set of disjoint blocks. So the segment must contain exactly $r$ elements of $X$, and each block $B_i$ must contain exactly one element of $X$ and two elements of $Y$. Condition~(b) shows that the segment must be contained in the initial segment of length $9k+\lfloor 3k/\ell\rfloor$ of the sequencing. Condition~(c) shows that $r\leq 3\ell$. So the segment contains at most $6\ell$ elements of $Y$ and at most $3\ell$ elements of $X$. Now consider the last element $y\in Y$ in our segment. Our choice of $y$ implies there is no block of the form $\{x',y',y\}$ with $x'\in X'$ and $y'\in Y'$, where $X\subseteq X$ and $Y'\subseteq Y$ are defined above. But the definition of the set $X'$ implies that the set of elements from $X$ in the segment are all contained in the set $X'$. Similarly, the set of elements $Y$ in the segment are all contained in the set $Y'\cup\{y\}$. So $y$ cannot be contained in a block in our segment, which gives us our required contradiction.
\end{proof}

\section{Comments}
\label{sec:comments}

This section provides a list of questions and open problems on the topics covered in this paper.

In Section~\ref{sec:PSTS} we exhibited examples of Steiner triple systems with $\ell$-good sequencings, where $\ell$ is large. The dual question is also interesting:
\begin{question}
\label{qn:badSTS}
Can we find an infinite family of STS$(n)$ with no $\ell$-good sequencings, where $\ell$ grows slowly with $n$?
\end{question}
\noindent
The greedy algorithms in Sections~\ref{sec:PSTS} and~\ref{sec:general} show that $\ell$ cannot grow too slowly: we must have $\ell=\Omega(n^{1/4})$. The Stinson--Veitch bound~\cite[Theorem~2.1]{StinsonVeitch_STS_sequencing} that we quoted in Section~\ref{sec:STS} shows that any family of Steiner triple systems we will work for $\ell> (n+2)/3$. Can this gap be closed? We may ask similar questions for Steiner quadruple systems, and indeed for any natural families of $(\ell,\cF)$-good sequences (such as those arising from block designs).

The families of examples arising from Question~\ref{qn:badSTS} would provide a limit on how well the greedy algorithms could perform for Steiner triple systems. More generally, we may ask whether these algorithms can be improved:
\begin{question}
\label{qn:greedy}
Are there algorithms to construct $(\ell,\cF)$-good sequences that work for significantly smaller values of $\ell$?
\end{question}
\noindent
We note that the range of values $\ell$ where our greedy algorithms work is essentially determined by the maximum size of the set of unfortunate elements. So either some technique for reducing the size of this set is needed, or a rather different algorithm is required. It would also be very interesting to see a  non-constructive (probabilistic) approach that shows the existence of $(\ell,\cF)$-good sequences for a wide range of values of $\ell$ and natural families $\cF$ of forbidden sequences; the counting argument of Kreher and Stinson~\cite[Theorem~2.1]{KreherStinson_AJC_STS} can be rephrased as a probabilistic argument for Steiner triple systems with $\ell=3$, but this approach does not seem to generalise easily to the case when $\ell>3$. There is a lovely labelling argument, due to Charlie Colbourn and described in~\cite[Section~2.3]{KreherStinson_AJC_STS}, which shows that a Steiner triple system is $3$-good. This argument was extended to Mendelsohn triple systems in~\cite[Section~4]{KreherStinsonVeitch_Mendelsohn}. Can this argument be extended to other situations?

It is interesting to explore adversarial constructions of partial $\ell$-good sequencings. For example, consider the following game. Fix a Steiner system on a point set $V$ of cardinality $n$. Let $\ell$ be chosen so that the STS does not have an $\ell$-good sequencing. (In particular, $\ell\geq 3$.) Alice and Bob take it in turns to specify the next point in a partial sequencing of the Steiner system (so not equal to a previously used point). A player loses as soon as the partial sequencing fails to be $\ell$-good.
\begin{theorem}
\label{thm:game}
Bob has a winning strategy when playing on the Steiner triple system $\mathrm{S}(2,3,2^r-1)$ based on the Hamming code.
\end{theorem}
\begin{proof}
Recall that the point set $V$ of our Steiner triple system is the set of non-zero binary vectors of length $r$. Blocks are of the form $\{x, y,z\}$ where $x,y,z\in V$ are pairwise distinct such that $x+y+x=0$.

Suppose Alice begins by playing $v\in V$. Bob chooses $u\in V\setminus\{v\}$ and responds with $v+u$. To every point $x$ that Alice plays, Bob responds with $x+u$. We claim this is a winning strategy for Bob.

First, note that if Alice ever plays $u$ she loses. For if her previous play was $x$, then the partial sequencing ends with $x,x+u,u$ and these points form a block. Without loss of generality, we assume that Alice never plays~$u$. Note that Bob also never plays $u$, as Alice never plays $0$. 

We claim that Bob's strategy is well defined. First, when Alice plays $x$ then $x\not=u$ and so $x+u$ is always non-zero, and so Bob's response lies in $V$. Now $V\setminus\{u\}$ is partitioned into pairs $\{x,x+u\}$ of size $2$. The partial sequencing after Bob's turn is always a union of such pairs. So Alice must reply with an element in a pair that has not been used, and Bob replies with the companion element in that pair. We conclude that Bob's reply is always an unused element, and so Bob's strategy is well defined. 

We now show that Bob's strategy wins. Suppose for a contradiction that the strategy loses. Alice plays $x\in V$ as her last point, and Bob replies with $x+u$ to lose. Let $\bs$ be the $\ell$-good partial sequencing before Bob's play (which ends with $x$), and let $\bt$ be the partial sequencing afterwards; so $\bt$ extends $\bs$ by $x+u$, and $\bt$ is not $\ell$-good. Since $\bt$ is not $\ell$-good, there exist points $y,z$ in $\bs$ such that $\{y,z,x+u\}$ is a block, and such that $y$, $z$ and $x+u$ occur in the last $\ell$ positions of $\bt$. Since $y$, $z$ and $x+u$ together form a block they are distinct, and $y+z+(x+u)=0$. Suppose, without loss of generality, that $y$ occurs before $z$ in the partial sequencing $\bs$. We claim that $\{y,z+u,x\}$ is a block which is contained in the last $\ell-1$ positions of $\bs$; this is sufficient for our required contradiction, since this shows that $\bs$ is not $\ell$-good. Certainly $y+(z+u)+x=y+z+(x+u)=0$. If $z+u=x$, then $y=y+z+(x+u)=0$, which contradicts the fact that $y\in V$. Similarly, if $z+u=y$ then $x=0$ which is again a contradiction. The points $x$ and $y$ are distinct since $x$ is the last point in $\bs$ and $y$ occurs before $z$ in $\bs$. So $y$, $z+u$ and $x$ are pairwise distinct and hence $\{y,z+u,x\}$ is a block. Now $y$ and $x$ both occur within the last $\ell-1$ positions of $\bs$ (since they are within the last $\ell$ positions of $\bt$). If $z=x$, then $y=y+(z+x)+(u+u)=u$, which contradicts the fact that neither Alice nor Bob play $u$. So $z$ is not the last element of $\bt$. Hence, since $z+u$ occurs just before or just after $z$ in $\bt$, we see that $z+u$ occurs within the last $\ell-1$ positions of $\bs$. This establishes our claim, and so we have a contradiction as required. 
\end{proof}
\begin{question}
\label{qn:game}
Are there other families of Steiner triple systems in which the winner of the game above can be determined? 
\end{question}
\noindent
We remark that we may still play this game for values of $\ell$ where there exists an $\ell$-good sequencing, by declaring the game to be a draw if the game ends with an $\ell$-good sequencing. For the STS based on the Hamming code, the proof of Theorem~\ref{thm:game} shows that Bob still wins when $\ell\geq 3$. We mention that Steiner systems have been linked with combinatorial games in another manner: see the recent paper of Irie~\cite{Irie}, which builds on Conway and Sloane's analysis of the hexad game~\cite{ConwaySloane}.

\begin{question}
\label{qn:SV_asymptotic}
For a positive integer $n$, define $L(n)$ to be the maximum value of $\ell$ such that there exists an $n$-point Steiner triple system with an $\ell$-good sequencing. What is $\overline{\lim}_{n\rightarrow\infty}L(n)/n$?
\end{question}
\noindent
The Stinson-Veitch bound together with the construction in Section~\ref{sec:STS} implies that
\[
\frac{1}{4}\leq \overline{\lim}_{n\rightarrow\infty}\frac{L(n)}{n}\leq \frac{1}{3}.
\]

\begin{question}
\label{qn:cyclic_SV_asymptotic}
For a positive integer $n$, define $L_{\text{cyc}}(n)$ to be the maximum value of $\ell$ such that there exists an $n$-point Steiner triple system with a cyclic $\ell$-good sequencing. What is $\overline{\lim}_{n\rightarrow\infty}L_{\text{cyc}}(n)/n$?
\end{question}
\noindent
Corollary~\ref{cor:SV_asymptotic} together with the construction in Section~\ref{sec:STS} implies that
\[
\frac{1}{4}\leq \overline{\lim}_{n\rightarrow\infty}\frac{L_{\text{cyc}}(n)}{n}\leq 0.329.
\]

The analogues of Questions~\ref{qn:SV_asymptotic} and~\ref{qn:cyclic_SV_asymptotic} for Steiner quadruple systems are interesting: bounds on the limit are provided by the constructions and bounds in Section~\ref{sec:SQS}. Indeed, it is natural to ask these questions for wider classes of structures, such as designs $S_\lambda(t,k,n)$ with $\lambda$, $t$ and $k$ fixed, or directed analogues such as Mendelsohn or directed triple systems.

\begin{question}
Are there natural families of Steiner triple systems where the maximum value of $\ell$ such that there exists an $\ell$-good sequencing can be determined?
\end{question}
Steiner triple systems based on finite projective or affine planes are of course particularly interesting in this context.

Finally, we ask a question about sequenceable Steiner triple systems, as defined in Section~\ref{sec:alspach}:
\begin{question}
Can the leading term of the bound in Theorem~\ref{thm:alspach} be improved?
\end{question}
\noindent
We comment that the leading term comes from the `key observation' of Alspach \emph{et al.}, so a new idea is needed to provide an improvement. We also note that an $n$-point Steiner triple system with $k$ pairwise disjoint blocks has $n\geq 3k$, so the current upper bound is of the right order of magnitude.

\paragraph{Acknowledgement.} Very many thanks to Doug Stinson, for inspiring discussions during a visit to Royal Holloway in September 2019 which initiated the research in this paper.

\end{document}